\documentclass[11pt]{amsart}
\usepackage{amsmath}
\usepackage{amssymb}
\usepackage{amsthm}
\usepackage{amscd}
\usepackage{enumerate}
\usepackage{verbatim}
\usepackage[all]{xy}

\theoremstyle{plain}
\newtheorem{thm}{Theorem}[section]
\newtheorem{prop}[thm]{Proposition}
\newtheorem{lem}[thm]{Lemma}
\newtheorem{cor}[thm]{Corollary}

\theoremstyle{definition}

\newtheorem{rmk}[thm]{Remark}

\newcommand{\Hom}{\mathrm{Hom}}

\newcommand{\Aut}{\mathrm{Aut}}

\newcommand{\kbar}{\bar{K}}

\newcommand{\okbar}{\mathcal{O}_{\bar{K}}}

\newcommand{\Acrys}{A_{\mathrm{crys}}}
\newcommand{\Ast}{\hat{A}_{\mathrm{st}}}

\newcommand{\okey}{\mathcal{O}_K}
\newcommand{\oc}{\mathcal{O}_{\mathbb{C}}}

\newcommand{\oen}{\mathcal{O}_N}

\newcommand{\Ker}{\mathrm{Ker}}

\newcommand{\Img}{\mathrm{Im}}
\newcommand{\Fil}{\mathrm{Fil}}

\newcommand{\Tor}{\mathrm{Tor}}
\newcommand{\Ext}{\mathrm{Ext}}

\newcommand{\cM}{\mathcal{M}}

\newcommand{\Gal}{\mathrm{Gal}}

\newcommand{\Frac}{\mathrm{Frac}}

\newcommand{\cmtil}{\tilde{\mathcal{M}}}
\newcommand{\Stil}{\tilde{S}_1}

\newcommand{\Tsta}{T_{\mathrm{st},\underline{\pi}}^{*}}

\newcommand{\oeln}{\mathcal{O}_{L_n}}
\newcommand{\oefn}{\mathcal{O}_{F_n}}
\newcommand{\tokbar}{\tilde{\mathcal{O}}_{\kbar}}

\newcommand{\beF}{\mathfrak{b}_F}

\newcommand{\beLn}{\mathfrak{b}_{L_n}}

\newcommand{\bekbar}{\mathfrak{b}_{\kbar}}

\newcommand{\aFj}{\mathfrak{a}_{F/K}^{j}}

\newcommand{\PD}{\mathrm{PD}}

\newcommand{\SG}{\mathfrak{S}}
\newcommand{\SGm}{\mathfrak{M}}

\newcommand{\Mod}{\mathrm{Mod}}

\newcommand{\uep}{\underline{\varepsilon}}
\newcommand{\upi}{\underline{\pi}}

\newcommand{\oef}{\mathcal{O}_F}

\newcommand{\bZ}{\mathbb{Z}}

\begin{document}

\title[Ramification bound of torsion semi-stable representations]{On a ramification
bound of torsion semi-stable representations over a local field}
\author{Shin Hattori}
\date{\today}
\email{shin-h@math.kyushu-u.ac.jp}
\address{Faculty of Mathematics, Kyushu University}
\thanks{Supported by 21st Century COE program ``Mathematics of Nonlinear
Structure via Singularity'' at Department of Mathematics, Hokkaido university, and by the JSPS International Training Program (ITP)}

\begin{abstract}
Let $p$ be a rational prime, $k$ be a perfect field of characteristic
 $p$, $W=W(k)$ be the ring of Witt vectors, $K$ be a finite totally
 ramified extension of $\Frac(W)$ of degree $e$ and $r$ be a
 non-negative integer satisfying $r<p-1$. In
 this paper, we prove the upper numbering ramification group
 $G_{K}^{(j)}$ for $j>u(K,r,n)$ acts trivially on the $p^n$-torsion semi-stable $G_K$-representations with Hodge-Tate weights in $\{0,\ldots,r\}$, where $u(K,0,n)=0$,
 $u(K,1,n)=1+e(n+1/(p-1))$ and
 $u(K,r,n)=1-p^{-n}+e(n+r/(p-1))$ for $1<r<p-1$.
\end{abstract}

\maketitle

\section{Introduction}\label{intro}

Let $p$ be a rational prime, $k$ be a perfect field of characteristic
 $p$, $W=W(k)$ be the ring of Witt vectors and $K$ be a finite totally
 ramified extension of $K_0=\Frac(W)$ of degree $e=e(K)$. We normalize the valuation $v_K$ of $K$ as
$v_K(p)=e$ and extend this to any algebraic closure of $K$. Let the maximal
ideal of $K$ be denoted by $m_K$, an algebraic closure of $K$ by $\kbar$
 and the absolute Galois group of $K$ by $G_K=\Gal(\kbar/K)$. Let $G_{K}^{(j)}$ denote the
$j$-th upper numbering ramification group in the sense of
\cite{F}. Namely, we put $G_{K}^{(j)}=G_{K}^{j-1}$, where the latter is
the upper numbering ramification group defined in \cite{S_CL}.

Consider a proper smooth scheme $X_K$ over $K$ and put
$X_{\kbar}=X_K\times_K\kbar$. Let $\mathcal{L}\supseteq \mathcal{L}'$ be $G_K$-stable $\mathbb{Z}_p$-lattices in the
$r$-th etale cohomology group
$H^{r}_{\mathrm{\acute{e}t}}(X_{\kbar},\mathbb{Q}_p)$ such that the quotient $\mathcal{L}/\mathcal{L}'$ is killed by $p^n$. In \cite{F}, Fontaine conjectured the upper
numbering ramification group $G_{K}^{(j)}$ acts trivially on the
$G_K$-modules $\mathcal{L}/\mathcal{L}'$ and $H^{r}_{\mathrm{\acute{e}t}}(X_{\kbar},\bZ/p^n\bZ)$ for $j>e(n+r/(p-1))$ if $X_K$
has good reduction.  For $e=1$ and $r<p-1$, this conjecture
was proved independently by himself (\cite{F2}, for $n=1$) and Abrashkin
(\cite{Ab}, for any $n$), using the theory of Fontaine-Laffaille
(\cite{FL}) and the
comparison theorem of Fontaine-Messing (\cite{FM}, see also \cite{B_comparison} and \cite{BrMes}) between the etale cohomology groups
of $X_K$ and the crystalline cohomology groups of the reduction of
$X_K$. From these results, they also showed some rareness of a proper
smooth scheme over $\mathbb{Q}$ with everywhere good reduction
(\cite[Th\'{e}or\`{e}me 1]{F2}, \cite[Section 7]{Ab2}). In fact, they proved this ramification
bound for the torsion crystalline representations of $G_K$ with Hodge-Tate weights in
$\{0,\ldots,r\}$ in the case where $K$ is absolutely unramified.

On the other hand, for a torsion semi-stable representation with Hodge-Tate weights
in the same range, a similar ramification
bound for $e=1$ and $n=1$ is obtained by Breuil (see \cite[Proposition
9.2.2.2]{BrMes}). He showed, assuming the Griffiths transversality which
in general does not hold, that if $e=1$ and $r<p-1$, then the ramification group $G_{K}^{(j)}$ acts
trivially on the mod $p$ semi-stable representations for
$j>2+1/(p-1)$. 

In this paper, we prove a ramification bound for the torsion semi-stable representations of $G_K$ with Hodge-Tate weights in $\{0,\ldots,r\}$ with no assumption on $e$ but under the assumption $r<p-1$. Let $\pi$ be a uniformizer of $K$, $E(u)\in W[u]$ be the Eisenstein polynomial of $\pi$ over $W$ and $S$ be the $p$-adic completion of the divided power envelope of $W[u]$ with respect to the ideal $(E(u))$. Consider a category $\Mod^{r,\phi,N}_{/S_\infty}$ of filtered $(\phi_r,N)$-modules over the ring $S$ and a $G_K$-module 
\[
\Tsta(\cM)=\Hom_{S,\mathrm{Fil}^r,\phi_r,N}(\cM,\hat{A}_{\mathrm{st},\infty})
\]
for $\cM\in\Mod^{r,\phi,N}_{/S_\infty}$, where $\hat{A}_{\mathrm{st},\infty}$ is a $p$-adic period ring (\cite{B_ss}). Then our main
theorem is the following.

\begin{thm}\label{main_br}
Let $r$ be a non-negative integer such that $r<p-1$. Let $\cM$ be an object of the category $\Mod^{r,\phi,N}_{/S_\infty}$ which is killed by $p^n$. Then the $j$-th upper numbering
 ramification group $G_{K}^{(j)}$ acts trivially on the $G_K$-module
 $\Tsta(\cM)$ for $j>u(K,r,n)$, where
\[
u(K,r,n)=\left\{
\begin{array}{lc}
0 & (r=0),\\
1+e(n+\frac{1}{p-1}) & (r=1),\\
1-\frac{1}{p^n}+e(n+\frac{r}{p-1}) & (1<r<p-1).
\end{array}
\right.
\]

\end{thm}
We can check that this bound is sharp for $r\leq 1$ (Remark \ref{sharp}). 

From this theorem and \cite[Proposition 1.3]{F}, we have the
following corollary.
\begin{cor}\label{discbound}
Let the notation be as in the theorem and $L$ be the finite
 extension of $K$ cut out by the $G_K$-module
 $\Tsta(\cM)$. Namely, the finite extension $L$ is defined by 
\[
G_L=\Ker(G_K\to\Aut(\Tsta(\cM))).
\]
Let $\mathfrak{D}_{L/K}$ denote the
 different of the extension $L/K$. Then we have the inequality
\[
v_K(\mathfrak{D}_{L/K})<u(K,r,n)
\]
for $r>0$ and $v_K(\mathfrak{D}_{L/K})=0$ for $r=0$.
\end{cor}

Combining these results with a theorem of Liu (\cite[Theorem 2.3.5]{Li}) or a theorem of Caruso (\cite[Th\'{e}or\`{e}me 1.1]{Ca2}), we will show the corollary below.

\begin{cor}\label{coret}
Let $r$ be a non-negative integer such that $r<p-1$. Then the same bounds as in Theorem \ref{main_br} and Corollary \ref{discbound} are also valid for the torsion $G_K$-modules of the following two cases:
\begin{enumerate}
\item the $G_K$-module $\mathcal{L}/\mathcal{L}'$, where $\mathcal{L}\supseteq \mathcal{L}'$ are $G_K$-stable $\bZ_p$-lattices in a semi-stable $p$-adic representation $V$ with Hodge-Tate weights in $\{0,\ldots,r\}$ such that $\mathcal{L}/\mathcal{L}'$ is killed by $p^n$.
\item the $G_K$-module $H^{r}_{\acute{e}t}(X_{\kbar},\bZ/p^n\bZ)$, where $X_K$ is a proper smooth algebraic variety over $K$ which has a proper semi-stable model over $\okey$ and $r$ satisfies $er<p-1$ for $n=1$ and $e(r+1)<p-1$ for $n>1$.
\end{enumerate}
\end{cor}

For the proof of Theorem \ref{main_br}, we basically follow a
 beautiful argument of Abrashkin (\cite{Ab}). We may assume $p\geq 3$
 and $r\geq 1$. Consider the finite Galois extension
\[
F_n=K(\pi^{1/p^n},\zeta_{p^{n+1}})
\]
of $K$ whose upper ramification is
bounded by $u(K,r,n)$. Let $L_n$ be the finite Galois extension of $F_n$
cut out by $\Tsta(\cM)|_{G_{F_n}}$. Then we bound the
ramification of $L_n$ over $K$. For this, we show that to study this
$G_{F_n}$-module we can use a
variant over a smaller coefficient ring $\Sigma$ of filtered
$(\phi_r,N)$-modules over $S$. In precise, we set
\[
\Sigma=W[[u,E(u)^p/p]].
\]
This ring $\Sigma$ is small enough for
the method of Abrashkin, in which he uses filtered modules of
Fontaine-Laffaille (\cite{FL}) whose coefficient ring is $W$, to work
also in the case where $K$ is absolutely ramified.

\ \linebreak
\noindent
{\bf Acknowledgments.} The author would like to pay his gratitude to Iku Nakamura and
Takeshi Saito for their warm encouragements. He wants to thank Manabu
Yoshida for kindly allowing him to include the proof of Proposition \ref{Yos}. He also wants to thank Ahmed
Abbes, Xavier Caruso, Takeshi Tsuji, and especially Victor Abrashkin and Yuichiro
Taguchi, for useful discussions and comments. Finally, he is very grateful to the referee for his or her valuable comments to improve this paper.

\section{Filtered $(\phi_r,N)$-modules of Breuil}

In this section, we recall the theory of filtered
$(\phi_r,N)$-modules over $S$ of Breuil, which is developed by himself and most
recently by Caruso and Liu (see for example \cite{B_ss}, \cite{Ca2},
\cite{Li}, \cite{CL}). In what follows, we always take the divided
power envelope of a $W$-algebra with the compatibility condition with the natural divided
power structure on $pW$.

Let $p$ be a rational prime and $\sigma$ be the Frobenius
endomorphism of $W$. We fix once and for all a uniformizer
$\pi$ of $K$ and a system $\{\pi_n\}_{n\in\mathbb{Z}_{\geq 0}}$ of
$p$-power roots of $\pi$ in $\kbar$ such that $\pi_0=\pi$ and $\pi_n=\pi_{n+1}^p$
for any $n$. Let $E(u)$ be the Eisenstein polynomial of
$\pi$ over $W$ and set $S=(W[u]^{\mathrm{PD}})^{\wedge}$, where $\PD$
means the divided power envelope and this is taken with respect to the ideal
$(E(u))$, and $\wedge$ means the $p$-adic
completion. The ring $S$ is endowed with the $\sigma$-semilinear endomorphism
$\phi: u \mapsto u^p$ and a natural filtration $\Fil^t S$ induced by the divided
power structure such that $\phi(\Fil^tS)\subseteq p^tS$ for $0\leq t\leq p-1$. We set $\phi_t=p^{-t}\phi|_{\Fil^t S}$ and
$c=\phi_1(E(u)) \in S^{\times}$. Let $N$ denote the $W$-linear
derivation on $S$ defined by the formula $N(u)=-u$. We also define a filtration, $\phi$,
$\phi_t$ and $N$ on $S_n=S/p^nS$ similarly. 

Let $r\in\{0,\ldots, p-2\}$ be an integer. Set $'\Mod^{r,\phi,N}_{/S}$ to
be the category consisting of the following data:

\begin{itemize}
\item an $S$-module $\cM$ and its $S$-submodule $\Fil^r \cM$ containing $\Fil^r S\cdot\cM$,
\item a $\phi$-semilinear map $\phi_r: \Fil^r\cM \to \cM$ satisfying
      \[\phi_r(s_r m)=\phi_r(s_r)\phi(m)\] for any $s_r \in \Fil^r S$
      and $m\in \cM$, where we set $\phi(m)=c^{-r}\phi_r(E(u)^r m)$,
\item a $W$-linear map $N:\cM\to\cM$ such that
\begin{itemize}
\item $N(sm)=N(s)m+sN(m)$ for any $s\in S$ and $m\in\cM$,
\item $E(u)N(\Fil^r\cM)\subseteq \Fil^r\cM$,
\item the following diagram is commutative:
\[
\begin{CD}
\Fil^r\cM @>{\phi_r}>> \cM\\
@V{E(u)N}VV @VV{cN}V\\
\Fil^r \cM @>>{\phi_r}> \cM,
\end{CD}
\]
\end{itemize} 
\end{itemize}
and the morphisms of $'\Mod^{r,\phi,N}_{/S}$ are defined to be the
$S$-linear maps preserving $\Fil^r$ and commuting with $\phi_r$ and
$N$. The category defined in the same way but dropping the data $N$ is
denoted by $'\Mod^{r,\phi}_{/S}$. These categories have obvious
notions of exact sequences. Let $\Mod^{r,\phi,N}_{/S_1}$ denote the full subcategory of
$'\Mod^{r,\phi,N}_{/S}$ consisting of $\cM$ such that $\cM$ is free of
finite rank over $S_1$ and generated as an $S_1$-module
by the image of $\phi_r$. 
We write $\Mod^{r,\phi,N}_{/S_\infty}$ for the smallest full
subcategory which contains $\Mod^{r,\phi,N}_{/S_1}$ and is stable under
extensions. We let $\Mod^{r,\phi,N}_{/S}$ denote the full subcategory consisting of $\cM$ such that
\begin{itemize}
\item the $S$-module $\cM$ is free of finite rank and generated by the
      image of $\phi_r$,
\item the quotient $\cM/\Fil^r\cM$ is $p$-torsion free.
\end{itemize}
We define full subcategories $\Mod^{r,\phi}_{/S_1}$,
$\Mod^{r,\phi}_{/S_\infty}$ and $\Mod^{r,\phi}_{/S}$ of
$'\Mod^{r,\phi}_{/S}$ in a similar way. For
$\hat{\cM}\in\Mod^{r,\phi,N}_{/S}$ ({\it resp.} $\Mod^{r,\phi}_{/S}$), the
quotient $\hat{\cM}/p^n\hat{\cM}$ has a natural
structure as an object of $\Mod^{r,\phi,N}_{/S_\infty}$ ({\it resp.} $\Mod^{r,\phi}_{/S_\infty}$).

For $p$-torsion objects, we also have the following categories. Consider the
$k$-algebra $k[u]/(u^{ep})\cong S_1/\Fil^pS_1$ and let this algebra be denoted by
$\Stil$. The algebra $\Stil$ is equipped with the natural filtration, $\phi$ and $N$
induced by those of $S$. Namely, $\Fil^t\Stil=u^{et}\Stil$,
$\phi(u)=u^p$ and $N(u)=-u$. Let $'\Mod^{r,\phi,N}_{/\Stil}$ denote the
category consisting of the following data:
\begin{itemize}
\item an $\Stil$-module $\cmtil$ and its $\Stil$-submodule $\Fil^r \cmtil$ containing $u^{er}\cmtil$,
\item a $\phi$-semilinear map $\phi_r: \Fil^r\cmtil \to \cmtil$,
\item a $k$-linear map $N:\cmtil\to\cmtil$ such that
\begin{itemize}
\item $N(sm)=N(s)m+sN(m)$ for any $s\in \Stil$ and $m\in\cmtil$,
\item $u^e N(\Fil^r\cmtil)\subseteq \Fil^r\cmtil$,
\item the following diagram is commutative:
\[
\begin{CD}
\Fil^r\cmtil @>{\phi_r}>> \cmtil\\
@V{u^e N}VV @VV{cN}V\\
\Fil^r \cmtil @>>{\phi_r}> \cmtil,
\end{CD}
\]
\end{itemize} 
\end{itemize}
and whose morphisms are defined as before. Its full subcategory
$\Mod^{r,\phi,N}_{/\Stil}$ is
defined by the following condition:
\begin{itemize}
\item As an $\Stil$-module, $\cmtil$ is free of finite rank and
      generated by the image of $\phi_r$.
\end{itemize}
We define categories $'\Mod^{r,\phi}_{/\Stil}$ and
$\Mod^{r,\phi}_{/\Stil}$ similarly. Then we can show as in the proof of \cite[Proposition 2.2.2.1]{B_ENS} that the natural functor 
$\cM\mapsto\cM/\Fil^pS\cdot\cM$ induces equivalences of categories
$T:\Mod^{r,\phi,N}_{/S_1}\to\Mod^{r,\phi,N}_{/\Stil}$ and $T_0:\Mod^{r,\phi}_{/S_1}\to\Mod^{r,\phi}_{/\Stil}$.

For $r=0$, let $\Mod^{\phi}_{/W_\infty}$ be the category consisting of the following data:
\begin{itemize}
\item a finite torsion $W$-module $\tilde{M}$,
\item a $\sigma$-semilinear automorphism $\phi:\tilde{M}\to\tilde{M}$.
\end{itemize}
Let $\kappa$ be the kernel of the natural surjection $S\to W$ defined by $u\mapsto 0$. Since $\Tor^{S}_{1}(\cM,S/\kappa S)=0$ for any $\cM\in\Mod^{0,\phi}_{/S_\infty}$, the proofs of \cite[Lemme 2.2.7, Proposition 2.2.8]{Ca2} work also for the category $\Mod^{0,\phi,N}_{/S_\infty}$ and we have a commutative diagram of categories
\[
\xymatrix{
\Mod^{0,\phi,N}_{/S_\infty} \ar[rd]\ar[r] & \Mod^{0,\phi}_{/S_\infty} \ar[d] \\
& \Mod^{\phi}_{/W_\infty},
}
\]
where the downward arrows and horizontal arrow are defined by $\cM\mapsto \cM/\kappa\cM$ and forgetting $N$ respectively and these three arrows are equivalences of categories.

Let $\Acrys$ and $\Ast$ be $p$-adic period
rings. These are constructed as follows. Put $\tokbar=\okbar/p\okbar$. Set $R$ to be the ring
\[
R=\varprojlim(\tokbar\gets\tokbar\gets\cdots),
\]
where every arrow is the $p$-power map. For an element
$x=(x_i)_{i\in\mathbb{Z}_{\geq 0}}\in R$ and an integer $n\geq 0$, we set
\[
x^{(n)}=\lim_{m\to\infty}\hat{x}_{n+m}^{p^m}\in\mathcal{O}_{\mathbb{C}},
\]
where $\hat{x}_i$ is a lift of $x_i$ in $\okbar$ and $\oc$ is the
$p$-adic completion of $\okbar$. Let $v_p$ denote the valuation of $\oc$
normalized as $v_p(p)=1$. Then the ring $R$ is a complete valuation ring whose
valuation of an element $x\in R$ is given by $v_R(x)=v_p(x^{(0)})$. We define a natural ring
homomorphism $\theta$ by
\begin{align*}
\theta: W(R)&\to \oc\\
(x_0,x_1,\ldots)&\mapsto \sum_{n\geq 0}p^n x_{n}^{(n)}.
\end{align*}
Then $\Acrys$ is the $p$-adic completion of the divided power envelope of
$W(R)$ with respect to the principal ideal $\Ker(\theta)$ and $\Ast$ is the $p$-adic
completion of the divided power polynomial ring 
$\Acrys\langle X\rangle$ over $\Acrys$. We set
$A_{\mathrm{crys},\infty}=\Acrys\otimes_W K_0/W$ and
$\hat{A}_{\mathrm{st},\infty}=\Ast\otimes_W K_0/W$. Put
$\underline{\pi}=(\pi_n)_{n\in\mathbb{Z}_{\geq 0}}\in R$, where we abusively let
$\pi_n$ denote the image of $\pi_n\in\okbar$ in $\tokbar$. These rings are considered
as $S$-algebras by the ring homomorphisms $S\to\Ast$ and $\Ast\to\Acrys$
which are defined by $u\mapsto[\underline{\pi}]/(1+X)$ and $X\mapsto 0$, respectively.
The ring $\Acrys$ is endowed with a natural filtration induced by the
divided power structure, a natural Frobenius endomorphism $\phi$ and the
$\phi$-semilinear map $\phi_t=p^{-t}\phi|_{\Fil^t\Acrys}$. With
these structures, $\Acrys$ and $A_{\mathrm{crys},\infty}$ are considered as objects of $'\Mod^{r,\phi}_{/S}$. Moreover, the
absolute Galois group $G_K$ acts naturally on these two rings. As for $\Ast$,
its filtration is defined by
\[
\Fil^t\Ast=\left\{\sum_{i\geq
0}a_i\frac{X^i}{i!}\ \left|\ a_i\in\Fil^{t-i}\Acrys,\ \underset{i\to\infty}{\lim}a_i=0\right.\right\}
\]
and the Frobenius structure of $\Acrys$ extends to $\Ast$ by 
\begin{align*}
\phi(X)&=(1+X)^p-1,\\
\phi_t&=p^{-t}\phi|_{\Fil^t\Ast}.
\end{align*}
We write $N$ also for the $\Acrys$-linear derivation on $\Ast$ defined by $N(X)=1+X$. The rings $\Ast$ and $\hat{A}_{\mathrm{st},\infty}$ are objects of
$'\Mod^{r,\phi,N}_{/S}$. The $G_K$-action on $\Acrys$ naturally extends to an action
on $\Ast$. Indeed, the action of $g\in G_K$ on $\Ast$ is defined by the formula
\[
g(X)=[\underline{\varepsilon}(g)](1+X)-1,
\]
where $g(\pi_n)=\varepsilon_n(g)\pi_n$ and 
$\underline{\varepsilon}(g)=(\varepsilon_n(g))_{n\in\mathbb{Z}_{\geq 0}}\in R$ with the abusive notation as above.

These rings have other descriptions, as follows. For an integer $n\geq 1$, put
$W_n=W/p^nW$ and let
$W_n(\tokbar)$ be the ring of Witt vectors
of length $n$ associated to $\tokbar$. We define a $W_n$-algebra
structure on $W_n(\tokbar)$ by twisting the natural $W_n$-algebra
structure by $\sigma^{-n}$. Then the natural ring homomorphism
\begin{align*}
\theta_n:W_n(\tokbar)&\to\okbar/p^n\okbar\\
(a_0,\ldots,a_{n-1})&\mapsto \sum_{i=0}^{n-1}p^{i}\hat{a}^{p^{n-i}}_{i},
\end{align*}
where $\hat{a}_i$ is a lift of $a_i$ in $\okbar$, is $W_n$-linear. Let
us denote $W^{\PD}_{n}(\tokbar)$ the divided power envelope of
$W_n(\tokbar)$ with respect to the ideal $\Ker(\theta_n)$. This
ring is considered as an
$S$-algebra by $u\mapsto [\pi_n]$. This
ring also has a natural filtration defined by the divided power
structure, and a natural $G_K$-module structure. The Frobenius endomorphism of the ring
of Witt vectors induces on this ring a $\phi$-semilinear Frobenius endomorphism, which is
denoted also by $\phi$. Then, by the $S$-linear transition maps
\begin{align*}
W^{\PD}_{n+1}(\tokbar)&\to W^{\PD}_{n}(\tokbar)\\
(a_0,\ldots,a_{n})&\mapsto ({a}_{0}^{p},\ldots,{a}_{n-1}^{p}),
\end{align*}
these $S$-algebras form a projective system compatible with all
the structures. Using this transition map, a $\phi$-semilinear map
\[
\phi_r:\Fil^r W^{\PD}_{n}(\tokbar)\to W^{\PD}_{n}(\tokbar)
\]
is defined by setting $\phi_r(x)$ to be the image of
$p^{-r}\phi(\hat{x})$, where $\hat{x}$ is a lift of $x$ in 
$\Fil^r W^{\PD}_{n+r}(\tokbar)$. By definition, the maps $\phi_r$
are also compatible with the transition maps. The $S$-algebra
$W^{\PD}_{n}(\tokbar)$ is considered as an object of
$'\Mod^{r,\phi}_{/S}$. Then we have a natural isomorphism in $'\Mod^{r,\phi}_{/S}$
\begin{align*}
\Acrys/p^n\Acrys&\to W^{\PD}_{n}(\tokbar)\\
(x_0,\ldots,x_{n-1})&\mapsto (x_{0,n},\ldots,x_{n-1,n}),
\end{align*}
where we set $x_i=(x_{i,k})_{k\in\mathbb{Z}_{\geq 0}}$ for $x_i\in R$.

Similarly, the divided
power polynomial ring $W^{\PD}_{n}(\tokbar)\langle X\rangle$
over
$W^{\PD}_{n}(\tokbar)$ is considered as an $S$-algebra by
$u\mapsto [\pi_n]/(1+X)$. This ring has a natural filtration coming from
the divided power structure. We define a $G_K$-action on this ring by
\[
g(X)=[\varepsilon_n(g)](1+X)-1.
\]
We also define a $\phi$-semilinear
Frobenius endomorphism, which we also write as $\phi$, by $\phi(X)=(1+X)^p-1$
and a $W^{\PD}_{n}(\tokbar)$-linear derivation $N$ by
$N(X)=1+X$. These rings form a projective system of
$S$-algebras compatible with all the structures by the transition maps
defined by the maps above and $X\mapsto X$. We define $\phi$-semilinear maps
\[
\phi_r:\Fil^r W^{\PD}_{n}(\tokbar)\langle X\rangle\to W^{\PD}_{n}(\tokbar)\langle X\rangle
\]
compatible with the transition maps as before. The
$S$-algebra $W^{\PD}_{n}(\tokbar)\langle X\rangle$ is considered
as an object of $'\Mod^{r,\phi,N}_{/S}$ and there exists a natural isomorphism in
$'\Mod^{r,\phi,N}_{/S}$
\begin{align*}
\Ast/p^n\Ast&\to W^{\PD}_{n}(\tokbar)\langle X\rangle\\
(x_0,\ldots,x_{n-1})&\mapsto (x_{0,n},\ldots,x_{n-1,n})\\
X&\mapsto X
\end{align*}
which is $G_K$-linear.

Put $K_n=K(\pi_n)$ and  $K_\infty=\cup_n K_n$. For
$\cM\in{}\Mod^{r,\phi,N}_{/S_\infty}$, we define a $G_K$-module $\Tsta(\cM)$ to be
\[
\Tsta(\cM)=\Hom_{S,\Fil^r,\phi_r,N}(\cM,\hat{A}_{\mathrm{st},\infty}).
\]
When $\cM$ is killed by $p^n$, we have a natural identification of $G_K$-modules
\[
\Tsta(\cM)=\Hom_{S,\Fil^r,\phi_r,N}(\cM,W^{\PD}_{n}(\tokbar)\langle X\rangle).
\]
Note that the $G_K$-module on the right-hand side is independent of the
choice of $\pi_k$ for $k>n$. Since the natural map
\begin{align*}
W^{\PD}_{n}(\tokbar)\langle X\rangle &\to
 W^{\PD}_{n}(\tokbar)\\
X&\mapsto 0
\end{align*}
is $G_{K_n}$-linear, we also have a $G_{K_n}$-linear
isomorphism (\cite[Lemme 2.3.1.1]{B_ss})
\[
\Tsta(\cM)|_{G_{K_n}}\to
\Hom_{S,\Fil^r,\phi_r}(\cM,W^{\PD}_{n}(\tokbar)).
\]
On the other hand, for $r=0$, the proof of \cite[Proposition 2.3.13]{Ca2} shows that the $G_K$-module $\Tsta(\cM)$ is unramified for any $\cM\in\Mod^{0,\phi,N}_{/S_\infty}$.

A variant of filtered $(\phi_r,N)$-modules over $S$ is also introduced
by Breuil and Kisin, and developed also
by Caruso and Liu (see for example \cite{Ki_BT}, \cite{Ki_Fcrys}, \cite{Li}, \cite{CL}). Put $\SG=W[[u]]$ and let $\phi:\SG\to\SG$ be the
$\sigma$-semilinear Frobenius endomorphism defined by $\phi(u)=u^p$. Let
$'\Mod^{r,\phi}_{/\SG}$ denote the category consisting of the following data:
\begin{itemize}
\item an $\SG$-module $\SGm$,
\item a $\phi$-semilinear map $\SGm\to\SGm$, which is denoted also by
      $\phi$, such that the cokernel of the map
      $1\otimes\phi:\phi^{*}\SGm\to\SGm$, where we set $\phi^{*}\SGm=\SG\otimes_{\phi,\SG}\SGm$, is killed by $E(u)^r$,
\end{itemize}
and whose morphisms are defined as before. The full subcategory of
$'\Mod^{r,\phi}_{/\SG}$ consisting of $\SGm$ such that $\SGm$ is free of
finite rank over $\SG/p\SG$ ({\it resp.} over $\SG$) is denoted by
$\Mod^{r,\phi}_{/\SG_1}$ ({\it resp.} $\Mod^{r,\phi}_{/\SG}$). We let
$\Mod^{r,\phi}_{/\SG_\infty}$ denote the smallest
full subcategory which contains $\Mod^{r,\phi}_{/\SG_1}$ and is stable under
extensions, as before. Then we have an exact functor (\cite[Proposition
2.1.2]{CL}, see also \cite[Proposition 1.1.11]{Ki_BT})
\[
\cM_{\SG_\infty}:\Mod^{r,\phi}_{/\SG_\infty}\to \Mod^{r,\phi}_{/S_\infty}.
\]
For $\SGm\in\Mod^{r,\phi}_{/\SG_\infty}$, the filtered $\phi_r$-module
$\cM=\cM_{\SG_\infty}(\SGm)$ over $S$ is defined as follows:
\begin{itemize}
\item $\cM=S\otimes_{\phi,\SG}\mathfrak{M}$,
\item
      $\Fil^r\cM=\Ker(\cM\overset{1\otimes\phi}{\to}S\otimes_{\SG}\mathfrak{M}\to (S/\Fil^r S)\otimes_{\SG}\mathfrak{M})$,
\item $\phi_r:\Fil^r \cM\overset{1\otimes\phi}{\to}\Fil^r S\otimes_{\SG}\mathfrak{M}\overset{\phi_r\otimes 1}{\to}S\otimes_{\phi,\SG}\mathfrak{M}=\cM$.
\end{itemize}
We write $\cM_{\SG}$ for the
functor $\Mod^{r,\phi}_{/\SG}\to\Mod^{r,\phi}_{/S}$ defined similarly.



\section{Filtered $\phi_r$-modules over $\Sigma$}

In this section, we define another variant $\Mod^{r,\phi}_{/\Sigma_\infty}$ of the category $\Mod^{r,\phi}_{/S_\infty}$ over a subring $\Sigma$ of the ring $S$, and prove that they are categorically equivalent.

Let $p$ be a rational prime and $r$ be an integer such that $0\leq
r<p-1$. Consider the $W$-algebra $\Sigma=W[[u,Y]]/(E(u)^p-pY)$ as in
\cite[Subsection 3.2]{B_ss}. We regard $\Sigma$ as a subring of $S$ by the map sending $Y$ to
$E(u)^p/p$. Then the
element $c=\phi_1(E(u))\in S^\times$ is contained in $\Sigma^\times$. We define on $\Sigma$ a
$\sigma$-semilinear Frobenius endomorphism $\phi$ by $\phi(u)=u^p$ and
$\phi(Y)=p^{p-1}c^p$. Put $\Fil^t\Sigma=(E(u)^t,Y)$ for 
$0\leq t\leq p-1$ and $\Fil^p\Sigma=(Y)$. Then we have
$\phi(\Fil^t\Sigma)\subseteq p^t\Sigma$ for $0\leq t\leq p-1$. We put
$\phi_t=p^{-t}\phi|_{\Fil^t\Sigma}$. We also set
$\Sigma_n=\Sigma/p^n\Sigma$ and put on this ring the natural structures induced
by those of $\Sigma$.

We define a category $'\Mod^{r,\phi}_{/\Sigma}$ of filtered
$\phi_r$-modules over $\Sigma$ to be the category consisting of the
following data:
\begin{itemize}
\item a $\Sigma$-module $M$ and its $\Sigma$-submodule $\Fil^rM$
      containing $\Fil^r\Sigma\cdot M$,
\item a $\phi$-semilinear map $\phi_r:\Fil^rM\to M$ satisfying
      $\phi_r(s_rm)=\phi_r(s_r)\phi(m)$ for any $s_r\in\Fil^r\Sigma$ and
      $m\in M$, where we set $\phi(m)=c^{-r}\phi_r(E(u)^rm)$,
\end{itemize}
and whose morphisms are defined in the same manner as $'\Mod^{r,\phi}_{/S}$. This
category has a natural notion of exact sequences. We define its full
subcategory $\Mod^{r,\phi}_{/\Sigma_1}$ to be the category consisting of $M$ which is
free of finite rank and generated by the image of $\phi_r$ as a
$\Sigma_1$-module. We also let
$\Mod^{r,\phi}_{/\Sigma_\infty}$ denote the smallest full subcategory of
$'\Mod^{r,\phi}_{/\Sigma}$ which contains $\Mod^{r,\phi}_{/\Sigma_1}$
and is stable under extensions. Moreover, we define a full subcategory $\Mod^{r,\phi}_{/\Sigma}$ of
$'\Mod^{r,\phi}_{/\Sigma}$ to be the category consisting of $M$ such
that
\begin{itemize}
\item the $\Sigma$-module $M$ is free of finite rank and generated by the
      image of $\phi_r$,
\item the quotient $M/\Fil^rM$ is $p$-torsion free.
\end{itemize}
Then we see that for $\hat{M}\in\Mod^{r,\phi}_{/\Sigma}$, the quotient
$\hat{M}/p^n\hat{M}$ is naturally considered as an object of $\Mod^{r,\phi}_{/\Sigma_\infty}$.

The natural ring isomorphism $\Sigma_1/\Fil^p\Sigma_1\cong\Stil$ defines
a functor $T_{0,\Sigma}:\Mod^{r,\phi}_{/\Sigma_1}\to\Mod^{r,\phi}_{/\Stil}$ by
$M\mapsto M/\Fil^p\Sigma_1\cdot M$. Then just as in the case of the functor $T_0:\Mod^{r,\phi}_{/S_1}\to\Mod^{r,\phi}_{/\Stil}$ (\cite[Proposition 2.2.2.1]{B_ENS}), we can show the following lemma.

\begin{lem}\label{ptoreq}
The functor $T_{0,\Sigma}:\Mod^{r,\phi}_{/\Sigma_1}\to\Mod^{r,\phi}_{/\Stil}$ is an equivalence of categories.
\end{lem}

On the other hand, \cite[Proposition 2.2.1.3]{B_ss} and
Nakayama's lemma show the following.

\begin{lem}\label{adaptedp}
Let $M$ be an object of $\Mod^{r,\phi}_{/\Sigma_1}$ of rank $d$ over
 $\Sigma_1$. Then there exists a basis $\{e_1,\ldots,e_d\}$ of $M$ such
 that $\Fil^rM=\Sigma_1 u^{r_1}e_1\oplus\cdots\oplus \Sigma_1 u^{r_d}e_d+\Fil^p\Sigma_1\cdot M$ for some
 integers $r_1,\ldots,r_d$ with $0\leq r_i\leq er$ for any $i$.
\end{lem}

Then we can show the following lemma just as in the proof of \cite[Lemme
2.3.1.3]{B_ss}.

\begin{lem}\label{ACexact}
The functor 
\[
M\mapsto \Hom_{\Sigma,\Fil^r,\phi_r}(M,A_{\mathrm{crys},\infty})
\]
from $\Mod^{r,\phi}_{/\Sigma_\infty}$ to the category of
 $G_{K_{\infty}}$-modules is exact.
\end{lem}

For $M\in\Mod^{r,\phi}_{/\Sigma_1}$, we can show as in the
case of the category $\Mod^{r,\phi}_{/S_1}$ that there is an
isomorphism of $G_{K_1}$-modules
\[
\Hom_{\Sigma,\Fil^r,\phi_r}(M,(\tokbar)^{\PD}) \to \Hom_{\Stil,\Fil^r,\phi_r}(T_{0,\Sigma}(M),\tokbar),
\]
where $\tokbar$ is considered as an object of
$'\Mod^{r,\phi}_{/\Stil}$ by the natural isomorphism
\[
(\tokbar)^{\PD}/\Fil^p(\tokbar)^{\PD}\to \tokbar.
\]
Thus \cite[Lemme 2.3.1.2]{B_ss} implies the following.

\begin{lem}\label{cardinality}
For $M\in\Mod^{r,\phi}_{/\Sigma_1}$, we have
\[
\#\Hom_{\Sigma,\Fil^r,\phi_r}(M,(\tokbar)^{\PD})=p^d,
\]
where $d=\dim_{\Sigma_1}M$.
\end{lem}

For the category $\Mod^{r,\phi}_{/\Sigma_\infty}$, we have the following lemma.

\begin{lem}\label{adaptedlem}
Let $M$ be in $\Mod^{r,\phi}_{/\Sigma_\infty}$.
Then there exists $\alpha_1,\ldots,\alpha_d\in \Fil^r M$ such
 that $\Fil^r M=\Sigma \alpha_1+\cdots +\Sigma \alpha_d+\Fil^p\Sigma\cdot M$ and
 the elements $e_1=\phi_r(\alpha_1),\ldots,e_d=\phi_r(\alpha_d)$ form a system of generators of $M$.
\end{lem}
\begin{proof}
By induction and Lemma \ref{adaptedp}, we may assume that there exists an exact sequence of the category $\Mod^{r,\phi}_{/\Sigma_\infty}$
\[
0\to M'\to M\to M''\to 0
\]
such that the lemma holds for $M'$ and $M''$. Let $\alpha'_1,\ldots,\alpha'_{l'}$ ({\it resp.} $\alpha''_1,\ldots,\alpha''_{l''}$) be elements of $\Fil^r M'$ ({\it resp.}  $\Fil^r M''$) as in the lemma. Let $\alpha_l\in\Fil^r M$ be a lift of $\alpha''_l$. Then the elements $\alpha'_1,\ldots,\alpha'_{l'},\alpha_1,\ldots,\alpha_{l''}$ satisfy the condition in the lemma for $M$.
\end{proof}

\begin{cor}\label{C}
Let $M$ be an object of $\Mod^{r,\phi}_{/\Sigma_\infty}$ and $C\in M_d(\Sigma)$ be a matrix satisfying
\[
(\alpha_1,\ldots,\alpha_d)=(e_1,\ldots,e_d)C
\]
with the notation of the previous lemma. Let $A$ be an object of $'\Mod^{r,\phi}_{/\Sigma}$. Then a $\Sigma$-linear homomorphism $f:M\to A$ preserving $\Fil^r$ also
 commutes with $\phi_r$ if and only if
\[
\phi_r(f(e_1,\ldots,e_d)C)=(f(e_1),\ldots,f(e_d)).
\]
\end{cor}
\begin{proof}
Suppose that the latter condition holds. Then we have $\phi_r(f(\alpha_i))=f(\phi_r(\alpha_i))$ for any $i$. 
We only have to check the equality $\phi_r\circ f=f\circ \phi_r$ on $\Fil^p\Sigma\cdot M$. Suppose that this equality holds on the submodule $p^{l+1}\Fil^p\Sigma\cdot M$. 
For $m\in M$, we can take $m'\in\Fil^p\Sigma\cdot M$ such that $E(u)^rm=\sum_i s_i\alpha_i+m'$. Let $s$ be in $\Fil^p\Sigma$. Then we have
\[
f(\phi_r(p^lsm))=p^l\phi_r(s)c^{-r}\sum_i\phi(s_i)f(\phi_r(\alpha_i))+p^l\phi_r(s)c^{-r}f(\phi_r(m')).
\]
Since $\phi_r(\Fil^p\Sigma)\subseteq p\Sigma$, this equals to $\phi_r(f(p^lsm))$ by assumption. Thus the lemma follows by induction.
\end{proof}

\begin{cor}\label{quotlem}
Let $M$ and $A$ be as above and $J\subseteq \Fil^rA$ be a $\Sigma$-submodule of $A$ such that
 $\phi_r(J)\subseteq J$. We can consider the $\Sigma$-module $A/J$
 naturally as an
 object of $'\Mod^{r,\phi}_{/\Sigma}$. Suppose that for any $x\in J$,
 there exists $t\in\mathbb{Z}_{\geq 0}$ such that $\phi_{r}^{t}(x)=0$. Then the natural homomorphism of abelian groups
\[
\Hom_{\Sigma,\Fil^r,\phi_r}(M,A)\to\Hom_{\Sigma,\Fil^r,\phi_r}(M,A/J)
\]
is an isomorphism.
\end{cor}
\begin{proof}
The proof is similar to \cite[Subsection 2.2]{Ab}. We consider the $\Sigma$-submodule $J$ as an object of the category $'\Mod^{r,\phi}_{/\Sigma}$ by putting $\Fil^rJ=J$. By devissage, it is enough to show that, for any $M\in\Mod^{r,\phi}_{/\Sigma_1}$, we have $\Ext_{'\Mod^{r,\phi}_{/\Sigma}}(M,J)=0$ and the map in the corollary is an isomorphism. For the first assertion, let 
\[
\xymatrix{
0 \ar[r] & J \ar[r] & \mathcal{E} \ar[r] & M \ar[r] & 0
}
\]
be an extension in the category $'\Mod^{r,\phi}_{/\Sigma}$. Let $e_i$, $\alpha_i$ and $C$ be as in Corollary \ref{C} such that $e_1,\ldots,e_d$ form a basis of $M$. Let $\hat{e}_i\in\mathcal{E}$ be a lift of $e_i\in M$. Then we have $(\hat{e}_1,\ldots,\hat{e}_d)C\in(\Fil^r\mathcal{E})^{\oplus d}$ and 
\[
\phi_r((\hat{e}_1,\ldots,\hat{e}_d)C)=(\hat{e}_1+\delta_1,\ldots,\hat{e}_d+\delta_d)
\]
for some $\delta_1,\ldots,\delta_d\in J$. On the other hand, there exists a unique $d$-tuple $(x_1,\ldots,x_d)\in J^{\oplus d}$ satisfying the equation
\[
\phi_r((\hat{e}_1+x_1,\ldots,\hat{e}_d+x_d)C)=(\hat{e}_1+x_1,\ldots,\hat{e}_d+x_d).
\]
Indeed, the $d$-tuple
\[
\sum_{i=0}^{t}(\phi_{r}^{i}(\delta_1),\ldots,\phi_{r}^{i}(\delta_d))\phi(C)\cdots\phi^{i-1}(C)\phi^i(C)
\]
is stable for sufficiently large $t$ by assumption and this limit gives a unique
 solution of the equation. Then we have
\[
(p(\hat{e}_1+x_1),\ldots,p(\hat{e}_d+x_d))=\phi_r(p(\hat{e}_1+x_1),\ldots,p(\hat{e}_d+x_d))\phi(C).
\]
Since the $d$-tuple on the left-hand side is contained in $J^{\oplus d}$, we see that this $d$-tuple is zero and $e_i\mapsto \hat{e}_i+x_i$ defines a section $M\to\mathcal{E}$. We can prove the second assertion similarly.
\end{proof}

Next we show that the two categories $\Mod^{r,\phi}_{/\Sigma_\infty}$ and $\Mod^{r,\phi}_{/S_\infty}$ are in fact equivalent. For $M\in \Mod^{r,\phi}_{/\Sigma_\infty}$, we associate to it an $S$-module $\cM$ by setting $\cM=S\otimes_\Sigma M$.
We also define its $S$-submodule $\Fil^r\cM$ by
\[
\Fil^r\cM=\Ker(\cM=S\otimes_\Sigma M\to S/\Fil^rS\otimes_\Sigma M/\Fil^rM\simeq M/\Fil^rM),
\]
where the last isomorphism is induced by the natural isomorphisms of $W$-algebras
\[
W[u]/(E(u)^r)\to \Sigma/\Fil^r\Sigma\to S/\Fil^rS.
\]
These associations induce two functors from $\Mod^{r,\phi}_{/\Sigma_\infty}$ to the category of $S$-modules, $M\mapsto \cM$ and $M\mapsto \Fil^r\cM$.
Since the rings $S$ and $W[u]/(E(u)^r)$ are $p$-torsion free, we have $\Tor^{\Sigma}_1(\Sigma_1,S)=\Tor^{\Sigma}_1(\Sigma_1,\Sigma/\Fil^r\Sigma)=0$ and thus $\Tor^{\Sigma}_1(M,S)=\Tor^{\Sigma}_1(M,\Sigma/\Fil^r\Sigma)=0$ for any $M\in \Mod^{r,\phi}_{/\Sigma_\infty}$. Hence we see that these two functors are exact.

We define $\phi_r:\Fil^r\cM \to \cM$ as follows. Note that $\Fil^rS\otimes_\Sigma M\subseteq \cM$ and $\Fil^r\cM$ is equal to $\Fil^rS\otimes_\Sigma M+\Img(S\otimes_\Sigma\Fil^rM\to \cM)$. Set $\phi'_r:\Fil^rS\otimes_\Sigma M\to \cM$ to be $\phi'_r=\phi_r\otimes\phi$.

\begin{lem}\label{phirwelldef}
The map $\phi\otimes \phi_r:S\otimes_\Sigma\Fil^rM\to \cM$ induces a $\phi$-semilinear map $\phi''_r:\Img(S\otimes_\Sigma\Fil^rM\to \cM)\to \cM$.
\end{lem}
\begin{proof}
Let $z=\sum_i s_i\otimes m_i$ be in $S\otimes_\Sigma\Fil^rM$ with $s_i\in S$ and $m_i\in \Fil^r M$. Let $\bar{z}$ be its image in $\cM$ and suppose that $\bar{z}=0$. Write $s_i=s'_i+s''_i$ with $s'_i\in \Sigma$ and $s''_i\in \Fil^pS$. Since we have an isomorphism $\cM/\Fil^rS\cdot\cM\simeq M/\Fil^r\Sigma\cdot M$, we can find elements $s^{(j)}\in \Fil^r \Sigma$ and $m^{(j)}\in M$ such that the equality $\sum_is'_im_i=\sum_j s^{(j)}m^{(j)}$ holds in $M$. Then we have
\[
0=\bar{z}=\sum_i 1\otimes s'_im_i+\sum_i s''_i\otimes m_i=\sum_j s^{(j)}\otimes m^{(j)}+\sum_i s''_i\otimes m_i
\]
in $\cM$. On the other hand, the element $(\phi\otimes\phi_r)(z)\in\cM$ is equal to
\[
\sum_j1\otimes\phi_r(s^{(j)}m^{(j)})+\sum_i\phi(s''_i)\otimes\phi_r(m_i).
\]
Since $\phi=p^r\phi_r$, this equals $\phi'_r(\sum_js^{(j)}\otimes m^{(j)}+\sum_i s''_i\otimes m_i)=0$.
\end{proof}

\begin{lem}\label{phirpatch}
The maps $\phi'_r$ and $\phi''_r$ patch together and define a $\phi$-semilinear map $\phi_r:\Fil^r\cM\to \cM$.
\end{lem}
\begin{proof}
Since $\phi'_r$ and $\phi''_r$ coincide on $\Img(\Fil^rS\otimes_\Sigma\Fil^rM\to\cM)$, it is enough to show that $1\otimes\phi_r(m)=\phi'_r(\sum_is_i\otimes m_i)$ for any $m\in\Fil^rM$, $s_i\in \Fil^rS$ and $m_i\in M$ satisfying $1\otimes m=\sum_i s_i\otimes m_i$ in $\cM$. As in the proof of Lemma \ref{phirwelldef}, the element $m$ can be written as $m=\sum_js^{(j)}m^{(j)}$ for some $s^{(j)}\in\Fil^r\Sigma$ and $m^{(j)}\in M$. By assumption, we have $\sum_is_i\otimes m_i=\sum_js^{(j)}\otimes m^{(j)}$ in $\Fil^rS\otimes_\Sigma M$. Hence the lemma follows.
\end{proof}

Then we see that this construction defines a functor $\cM_{\Sigma_\infty}:\Mod^{r,\phi}_{/\Sigma_\infty}\to\Mod^{r,\phi}_{/S_\infty}$.

\begin{lem}
The functor $\cM_{\Sigma_\infty}$ induces an equivalence of categories $\Mod^{r,\phi}_{/\Sigma_1}\to\Mod^{r,\phi}_{/S_1}$.
\end{lem}
\begin{proof}
Consider the diagram of functors
\[
\xymatrix{
\Mod^{r,\phi}_{/\Sigma_1}\ar[r]^{\cM_{\Sigma_\infty}}\ar[dr]_{T_{0,\Sigma}} & \Mod^{r,\phi}_{/S_1}\ar[d]^{T_0} \\
 & \Mod^{r,\phi}_{/\Stil}.
}
\]
From the definition, we see that this diagram is commutative. By Lemma \ref{ptoreq}, the downward arrows are equivalences of categories. Thus the lemma follows.
\end{proof}

Then a devissage argument as in \cite[Proposition 1.1.11]{Ki_BT} shows the following corollary.

\begin{cor}\label{fullfaith}
The functor $\cM_{\Sigma_\infty}:\Mod^{r,\phi}_{/\Sigma_\infty}\to\Mod^{r,\phi}_{/S_\infty}$ is fully faithful.
\end{cor}

To show the essential surjectivity of the functor $\cM_{\Sigma_\infty}$, we define another functor $M_{\SG_\infty}:\Mod^{r,\phi}_{/\SG_\infty}\to \Mod^{r,\phi}_{/\Sigma_\infty}$ which is defined in a similar way to the functor $\cM_{\SG_\infty}:\Mod^{r,\phi}_{/\SG_\infty}\to \Mod^{r,\phi}_{/S_\infty}$. For an $\mathfrak{S}$-module
$\mathfrak{M}$ in $\Mod^{r,\phi}_{/\SG_\infty}$, we associate to it a $\Sigma$-module $M\in{}'\Mod^{r,\phi}_{/\Sigma}$ as follows:
\begin{itemize}
\item $M=\Sigma\otimes_{\phi,\SG}\mathfrak{M}$,
\item $\Fil^rM=\Ker(M\overset{1\otimes\phi}{\to}\Sigma\otimes_{\SG}\mathfrak{M}\to (\Sigma/\Fil^r\Sigma)\otimes_{\SG}\mathfrak{M})$,
\item $\phi_r:\Fil^r M\overset{1\otimes\phi}{\to}\Fil^r\Sigma\otimes_{\SG}\mathfrak{M}\overset{\phi_r\otimes 1}{\to}\Sigma\otimes_{\phi,\SG}\mathfrak{M}=M$.
\end{itemize}
We can check that this defines an exact functor
$\Mod^{r,\phi}_{/\SG_\infty}\to\Mod^{r,\phi}_{/\Sigma_\infty}$ as in
the proof of \cite[Proposition 1.1.11]{Ki_BT}. We let this functor be denoted by
$M_{\SG_\infty}$.

\begin{lem}\label{comm}
The diagram of functors
\[
\xymatrix{
\Mod^{r,\phi}_{/\SG_\infty}\ar[r]^{M_{\SG_\infty}}\ar[dr]_{\cM_{\SG_\infty}} & \Mod^{r,\phi}_{/\Sigma_\infty}\ar[d]^{\cM_{\Sigma_\infty}} \\
 & \Mod^{r,\phi}_{/S_\infty}
}
\]
is commutative.
\end{lem}
\begin{proof}
For $\SGm\in\Mod^{r,\phi}_{/\SG_\infty}$, put $M=M_{\SG_\infty}(\SGm)$ and $\cM=\cM_{\SG_\infty}(\SGm)$. Then $\cM=S\otimes_\Sigma M$ as an $S$-module. Let $\Fil^r\cM$ and $\phi_r:\Fil^r\cM\to\cM$ denote the filtration and Frobenius structure defined by the functor $\cM_{\SG_\infty}$. We also let $\hat{\Fil^r}\cM$ and $\hat{\phi}_r:\hat{\Fil^r}\cM\to\cM$ denote those defined by $\cM_{\Sigma_\infty}$.

The $S$-module $\Fil^r\cM$ contains $\hat{\Fil^r}\cM$. Conversely, let $z$ be an element of $\Fil^r\cM$. Note that $\Fil^pS\cdot\cM\subseteq\hat{\Fil^r}\cM$. Thus, to show $z\in \hat{\Fil^r}\cM$, we may assume that $z\in\Img(M\to\cM)$. Then the commutative diagram whose right vertical arrow is an isomorphism
\[
\begin{CD}
M=\Sigma\otimes_{\phi,\mathfrak{S}}\mathfrak{M} @>{1\otimes\phi}>>
 \Sigma\otimes_{\mathfrak{S}}\mathfrak{M} @>>>
 \Sigma/\Fil^r\Sigma\otimes_{\mathfrak{S}}\mathfrak{M} \\
@VVV @VVV @VVV\\
\cM=S\otimes_{\phi,\mathfrak{S}}\mathfrak{M} @>{1\otimes\phi}>>
 S\otimes_{\mathfrak{S}}\mathfrak{M} @>>>
 S/\Fil^rS\otimes_{\mathfrak{S}}\mathfrak{M} 
\end{CD}
\]
implies that $z\in\Img(\Fil^rM\to\Fil^r\cM)\subseteq \hat{\Fil^r}\cM$ and hence $\Fil^r\cM=\hat{\Fil^r}\cM$. From the definition, we also can show $\phi_r=\hat{\phi}_r$. This implies the lemma. 
\end{proof}

\begin{prop}\label{SSigeq}
The functor $\cM_{\Sigma_\infty}:\Mod^{r,\phi}_{/\Sigma_\infty}\to\Mod^{r,\phi}_{/S_\infty}$ is an equivalence of categories.
\end{prop}
\begin{proof}
Since the functor $\cM_{\SG_\infty}$ is an equivalence of categories for $p\geq 3$ (\cite[Theorem 2.3.1]{CL}), Corollary \ref{fullfaith} and Lemma \ref{comm} imply the proposition in this case. For $r=0$, put $\kappa_\Sigma=\kappa\cap\Sigma$, where $\kappa=\Ker(S\to W)$. Then, by using a natural isomorphism $\Sigma \simeq W[[u,u^{ep}/p]]$, we can show that the functor $M\mapsto M/\kappa_\Sigma M$ defines an equivalence of categories $\Mod^{0,\phi}_{/\Sigma_\infty}\to \Mod^{\phi}_{/W_\infty}$, as in the case of the category $\Mod^{0,\phi}_{/S_\infty}$. Since the diagram
\[
\xymatrix{
\Mod^{0,\phi}_{/\Sigma_\infty}\ar[rd]\ar[r]^{\cM_{\Sigma_\infty}} & \Mod^{0,\phi}_{/S_\infty}\ar[d]\\
& \Mod^{\phi}_{/W_\infty}
}
\]
is commutative and the downward arrows are equivalences of categories, the proposition follows also for $p=2$.
\end{proof}

\begin{rmk}
We can also define a fully faithful functor $\cM_{\Sigma}:\Mod^{r,\phi}_{/\Sigma}\to\Mod^{r,\phi}_{/S}$ in a similar way to $\cM_{\Sigma_\infty}$ and prove that this is an equivalence of categories. Indeed, the claim for $p\geq 3$ follows from \cite[Theorem 2.2.1]{CL}. Let $\cM$ be in $\Mod^{0,\phi}_{/S}$ and $e_1,\ldots,e_d$ be a basis of $\cM$ over $S$. Let $C\in GL_d(S)$ be the matrix such that
\[
\phi(e_1,\ldots,e_d)=(e_1,\ldots,e_d)C.
\]
Then the elements $\phi(e_1),\ldots,\phi(e_d)$ also form a basis of $\cM$ and 
\[
\phi(\phi(e_1),\ldots,\phi(e_d))=(\phi(e_1),\ldots,\phi(e_d))\phi(C).
\]
Since $\phi(S)\subseteq \Sigma$, the $\Sigma$-module $M$ defined by $M=\Sigma\phi(e_1)\oplus\cdots\oplus\Sigma\phi(e_d)$ is stable under $\phi$. Hence we see that $M\in\Mod^{0,\phi}_{/\Sigma}$ and $\cM=\cM_\Sigma(M)$.
\end{rmk}

\begin{prop}\label{vsTcrys}
Let $M$ be an object of $\Mod^{r,\phi}_{/\Sigma_\infty}$ and set $\cM=\cM_{\Sigma_\infty}(M)$. Then there
 exists a natural isomorphism of $G_{K_\infty}$-modules
 \[
\Hom_{\Sigma,\Fil^r,\phi_r}(M,A_{\mathrm{crys},\infty})\to \Hom_{S,\Fil^r,\phi_r}(\cM,A_{\mathrm{crys},\infty}).
\]
Moreover, this induces for any $n$ an isomorphism of $G_{K_n}$-modules
\[
\Hom_{\Sigma,\Fil^r,\phi_r}(M,W_{n}^{\PD}(\tokbar))\to \Hom_{S,\Fil^r,\phi_r}(\cM,W_{n}^{\PD}(\tokbar)).
\]
\end{prop}
\begin{proof}
By definition, $\cM=S\otimes_{\Sigma}M$ and we have a natural isomorphism
\[
\Hom_{\Sigma}(M,A_{\mathrm{crys},\infty})\to\Hom_S(\cM,A_{\mathrm{crys},\infty}).
\]
From the definition, we can check that this isomorphism induces the map in the proposition, which is injective. To prove the bijectivity, by devissage we may assume that
 $pM=0$. Then both sides of this injection have the same
 cardinality by Lemma \ref{cardinality} and the first assertion follows. Since the sequence 
\[
\xymatrix{
0 \ar[r] & W_{n}^{\PD}(\tokbar) \ar[r] & A_{\mathrm{crys},\infty} \ar[r]^{p^n} & A_{\mathrm{crys},\infty} \ar[r] & 0
}
\]
of the category $'\Mod^{r,\phi}_{/\Sigma}$ is exact, the first assertion implies the second one.
\end{proof}


\section{A method of Abrashkin}

In this section, we study the $G_{K_n}$-module $\Hom_{\Sigma,\Fil^r,\phi_r}(M,W_{n}^{\PD}(\tokbar))$ following
Abrashkin (\cite{Ab}).

Let $p$ and $0\leq r<p-1$ be as before. We fix a system of $p$-power roots of unity
$\{\zeta_{p^n}\}_{n\in\mathbb{Z}_{\geq 0}}$ in $\kbar$ such that $\zeta_p\neq 1$ and $\zeta_{p^{n}}=\zeta_{p^{n+1}}^{p}$ for any $n$, and set an element $\underline{\varepsilon}$ of $R$ to be 
$(\zeta_{p^n})_{n\in\mathbb{Z}_{\geq 0}}$. 
Then the elements $[\underline{\varepsilon}]-1$ and
$[\underline{\varepsilon}^{1/p}]-1$ are topologically nilpotent in $W(R)$. The
element of $W(R)$
\[
t=([\uep]-1)/([\uep^{1/p}]-1)=1+[\uep^{1/p}]+[\uep^{1/p}]^2+\cdots+[\uep^{1/p}]^{p-1}
\]
is a generator of the principal ideal $\Ker(\theta)$. We define an element $a\in W(R)$ to be
\[
a=\left\{\begin{array}{ll}
\sum_{k=1}^{p-2}p^{-1}((-1)^{p-1-k}{}_{p-1}C_k-1)[\uep^{1/p}]^k &\quad (p\geq 3) \\
-1 &\quad (p=2),
\end{array}\right.
\]
where ${}_{p-1}C_k=(p-1)!/(k!(p-1-k)!)$ is the binomial
coefficient. Note that the coefficient of $[\uep^{1/p}]^k$ in each term
is an integer. The element $a$ is invertible in the ring $W(R)$, since $\theta(a)=(\zeta_p-1)^{p-1}/p\in\mathcal{O}_{\mathbb{C}}^{\times}$ and the ideal $\Ker(\theta)$ is topologically nilpotent in $W(R)$.

The element
$Z=([\uep]-1)^{p-1}/p$ of $\Acrys$ is
topologically nilpotent and we have $\phi(t)=p(Z-\phi(a))$. Consider the formal power series ring $W(R)[[u']]$ with the $(t,u')$-adic topology and the continuous ring homomorphism $W(R)[[u']]\to \Acrys$ which sends $u'$ to $Z$. Let $\hat{A}$ denote the image of this homomorphism. Then we see that the ring $\hat{A}$ is $(t,Z)$-adically complete. Since we have $Z=at^{p-1}+t^p/p$, the element $t^p/p$ of $\Acrys$ is contained in the subring
 $\hat{A}$ and
 topologically nilpotent in this subring. Hence we can consider the ring
$\hat{A}$ as a
$\Sigma$-algebra by $u\mapsto [\upi]$. Put
$\Fil^i \hat{A}=(t^i,Z)$ for $0\leq i \leq p-1$. The
Frobenius endomorphism $\phi$ of $\Acrys$ preserves $\hat{A}$ and satisfies
$\phi(\Fil^i\hat{A})\subseteq p^i\hat{A}$ for $0\leq i\leq p-1$. Set
$\phi_r=p^{-r}\phi|_{\Fil^r\hat{A}}$. Then we can consider the ring
$\hat{A}$ also as an object of the category $'\Mod^{r,\phi}_{/\Sigma}$. Put $\hat{A}_n=\hat{A}/p^n\hat{A}$ and 
$\hat{A}_\infty=\hat{A}\otimes_W K_0/W$. We include here a proof of the following lemma stated in \cite[Subsection
3.2]{Ab}.

\begin{lem}\label{quotZ}
The natural inclusion $W(R)\to\hat{A}$ induces isomorphisms of $W(R)$-algebras
$W(R)/(([\uep]-1)^{p-1})\to \hat{A}/(Z)$ and $W_n(R)/(([\uep]-1)^{p-1})\to \hat{A}_n/(Z)$.
\end{lem}
\begin{proof}
For a subring $B$ of $\Acrys$, put
\[
I^{[s]}B=\{x\in B\ |\ \phi^i(x)\in\Fil^s\Acrys\text{ for any } i\}
\]
as in \cite[Subsection 5.3]{F_period}. Then we have
 $I^{[s]}W(R)=([\uep]-1)^sW(R)$ and the natural ring homomorphism
\[
W(R)/I^{[s]}W(R) \to \Acrys/I^{[s]}\Acrys
\]
is an injection (\cite[Proposition 5.1.3, Proposition
 5.3.5]{F_period}). Since the element $Z$ is contained in the ideal
 $I^{[p-1]}\Acrys$, this injection factors as
\[
W(R)/I^{[p-1]}W(R) \to \hat{A}/(Z)\to\Acrys/I^{[p-1]}\Acrys.
\]
Hence the former arrow is an isomorphism and the lemma follows.
\end{proof}

Therefore $\hat{A}/\Fil^r\hat{A}$ is $p$-torsion free and $p^n\Fil^r\hat{A}=\Fil^r\hat{A}\cap p^n\hat{A}$. Thus we can also consider $\hat{A}_n$ and 
$\hat{A}_\infty$ as objects of the category $'\Mod^{r,\phi}_{/\Sigma}$. The absolute Galois group
$G_{K_\infty}$ acts naturally on these $\Sigma$-modules.

\begin{lem}\label{A1}
We have a natural decomposition as an $R$-module
\[
\hat{A}_1=R/(t^p)\oplus (Z).
\]
\end{lem}
\begin{proof}
Consider the natural inclusion $W(R)\to \hat{A}$. We claim that
 this induces an injection $R/(t^p)\to \hat{A}_1$.
Let $x$ be in the ring $R$. If
 the element $[x]\in W(R)$ is contained in $p\hat{A}$, then its image in
 $\Acrys/p\Acrys$ is zero. We have an isomorphism of $R$-algebras
\[
R[Y_1,Y_2,\ldots]/(t^p,Y_{1}^{p},Y_{2}^{p},\ldots)\to\Acrys/p\Acrys
\]
which sends $Y_i$ to the image of $t^{p^i}/p^i!$. Thus the element $x$ is contained in the ideal $(t^p)$. Conversely, if $v_R(x)\geq p$, then we have
\[
[x]=w([\uep]-1)^{p-1}+pw'
\]
for some $w,w'\in W(R)$ and this implies $[x]\in p\hat{A}$. Now we have the commutative diagram of $R$-algebras
\[
\xymatrix{
R/(t^p)\ar[r]\ar[rd]_{f} & \hat{A}_1\ar[d] \\
 & \hat{A}_1/(Z)
}
\]
and the map $f:R/(t^p)\to \hat{A}_1/(Z)$ is an isomorphism by Lemma \ref{quotZ}. Hence the lemma follows.
\end{proof}

Since $r<p-1$, from this lemma we can show the following lemma as in the proof of \cite[Lemme 2.3.1.3]{B_ss}.

\begin{lem}\label{Aexact}
The functor
\[
M\mapsto \Hom_{\Sigma,\Fil^r,\phi_r}(M,\hat{A}_\infty)
\]
from $\Mod^{r,\phi}_{/\Sigma_\infty}$ to the category of $G_{K_\infty}$-modules
 is exact.
\end{lem}

\begin{cor}\label{AnAcrys}
For any $M\in\Mod^{r,\phi}_{/\Sigma_\infty}$, the natural map
\[
\Hom_{\Sigma,\Fil^r,\phi_r}(M,\hat{A}_\infty)\to \Hom_{\Sigma,\Fil^r,\phi_r}(M,A_{\mathrm{crys},\infty})
\]
is an isomorphism of $G_{K_\infty}$-modules. Moreover, for any $n$, we have an isomorphism of $G_{K_\infty}$-modules
\[
\Hom_{\Sigma,\Fil^r,\phi_r}(M,\hat{A}_n)\to \Hom_{\Sigma,\Fil^r,\phi_r}(M,A_{\mathrm{crys}}/p^nA_{\mathrm{crys}}).
\]
\end{cor}
\begin{proof}
Let us prove the first assertion. By Lemma \ref{ACexact} and Lemma \ref{Aexact}, we may assume
 $pM=0$. Consider the commutative diagram of rings
\[
\xymatrix{
\hat{A}_1\ar[r]\ar[rd] & \Acrys/p\Acrys\ar[d]\\
& R/(t^{p-1})
}
\]
whose downward arrows are defined by modulo $\Fil^{p-1}$ of the rings $\hat{A}_1$ and
 $\Acrys/p\Acrys$, respectively. Since $r<p-1$, we have
 $\phi_r(\Fil^{p-1}\hat{A}_1)=0$ and similarly for the ring
 $\Acrys/p\Acrys$. Thus these two surjections induce on the
 ring $R/(t^{p-1})$ the same
 structure of a filtered $\phi_r$-module over $\Sigma$. By Corollary \ref{quotlem}, we have a commutative diagram
\[
\xymatrix{
\Hom_{\Sigma,\Fil^r,\phi_r}(M,\hat{A}_1)\ar[r]\ar[rd] &
 \Hom_{\Sigma,\Fil^r,\phi_r}(M,\Acrys/p\Acrys)\ar[d]\\
& \Hom_{\Sigma,\Fil^r,\phi_r}(M,R/(t^{p-1}))
}
\]
whose downward arrows are isomorphisms. This concludes the proof of the first assertion. Since we have an exact sequence
\[
\xymatrix{
0 \ar[r] & \hat{A}_n \ar[r] & \hat{A}_\infty \ar[r]^{p^n} & \hat{A}_\infty \ar[r] & 0
}
\]
in the category $'\Mod^{r,\phi}_{/\Sigma}$, the second assertion follows.
\end{proof}

Since the ideal $(Z)$ of $\hat{A}_n$ satisfies the condition of Corollary
\ref{quotlem}, the $\Sigma$-algebra $\hat{A}_n/(Z)$ is naturally considered as
an object of $'\Mod^{r,\phi}_{/\Sigma}$. We also give the ring
$W_n(R)/(([\uep]-1)^{p-1})$ the structures of a
$\Sigma$-algebra and a filtered
$\phi_r$-module over $\Sigma$ induced from those of $\hat{A}_n/(Z)$ by the isomorphism in Lemma \ref{quotZ}. The map 
\[
\Sigma\to W_n(R)/(([\uep]-1)^{p-1})
\]
 sends the
element $u\in\Sigma$ to the image of 
$[\upi]$ in the ring on the right-hand side. Put $v=t/E([\upi])\in W(R)^\times$. As for the
element $Y\in\Sigma$, the equality
\[
Y=-av^{-1}E([\underline{\pi}])^{p-1} + v^{-p}Z
\]
holds in $\hat{A}$. Hence the above homomorphism sends the element $Y$ to the
image of $-av^{-1}E([\underline{\pi}])^{p-1}$.

Consider the surjective ring homomorphism
\begin{align*}
R\to \tokbar \\
x=(x_0,x_1,\ldots) \mapsto x_n
\end{align*}
and the induced surjection $\beta_n: W_n(R)\to W_n(\tokbar)$. Let 
\[
J=\{(x_0,\ldots,x_{n-1})\in W_n(R)\mid v_R(x_i)\geq p^n\text{ for any
}i\}
\]
be the kernel of the latter surjection.

\begin{lem}\label{ker}
The ideal $J$ is contained in the ideal $(([\uep]-1)^{p-1})$ of the
 ring $W_n(R)$.
\end{lem}
\begin{proof}
Write the element $([\uep]-1)^{p-1}$ also as $x=(x_0,\ldots,x_{n-1})\in W_n(R)$ 
with $v_R(x_0)=p$. Take an element $z=(z_0,\ldots,z_{n-1})$ of the ideal
 $J$. We construct
$y\in W_n(R)$ such that $xy=z$. By induction, it is
enough to show that if $z_0=\cdots=z_{i-1}=0$ for some $0\leq i\leq n-1$
and $(x_0,\ldots,x_i)(0,\ldots,0,y_i)=(0,\ldots,0,z_i)$ in $W_{i+1}(R)$, then
$x(0,\ldots,0,y_i,0,\ldots,0)\in J$. Let us write this element as $(0,\ldots,0,w_i,\ldots,w_{n-1})$ with $w_i=z_i$. We have 
$v_R(y_i)\geq p^n-p^{i+1}$. 
In the ring of Witt vectors
 $W_n(\mathbb{F}_p[X_0,\ldots,X_{n-1},Y_0,\ldots,Y_{n-1}])$, the $k$-th
 entry of the vector
\[
(X_0,\ldots,X_{n-1})(0,\ldots,0,Y_i,0,\ldots,0)
\]
is $X_{k-i}^{p^i}Y_{i}^{p^{k-i}}$ for any $k\geq i$. Thus we have
 $v_R(w_k)\geq p^n$.
\end{proof}

Note that the elements
$[\zeta_{p^n}]-1$ and $[\zeta_{p^{n+1}}]-1$ are nilpotent in
$W_n(\tokbar)$. By the above lemma, we have an isomorphism of rings
\[
W_n(R)/(([\uep]-1)^{p-1}) \to W_n(\tokbar)/(([\zeta_{p^n}]-1)^{p-1}).
\]
We let $\bar{A}_{n,p-1}$ denote the ring on the right-hand side and give the ring $\bar{A}_{n,p-1}$ the structure of a
filtered $\phi_r$-module over $\Sigma$ induced by this isomorphism.

For an algebraic extension $F$ of $K$, we put
\[
\mathfrak{b}_F=\{ x\in\oef\mid v_K(x)>er/(p-1)\}.
\]
Note that the ring $\oef/\mathfrak{b}_F$ is killed by $p$. We consider the ring of Witt vectors $W_n(\oef/\beF)$ as a $W_n(\oef)$-algebra by the natural ring surjection $W_n(\oef)\to W_n(\oef/\beF)$ and as a $W_n$-algebra by twisting the natural action by $\sigma^{-n}$, as before. For a ring $B$ and its ideal $I$, we define an ideal $W_n(I)$ of the ring $W_n(B)$ to be
\[
W_n(I)=\{ (x_0,\ldots,x_{n-1})\in W_n(B)\mid x_i\in I\text{ for any }i\}.
\]

Put $F_{n}=K_n(\zeta_{p^{n+1}})$. For an algebraic extension
$F$ of $F_n$ in $\kbar$, the
elements $[\zeta_{p^n}]-1$ and $[\zeta_{p^{n+1}}]-1$ of $W_n(m_F)$ are topologically
nilpotent non-zero divisors in $W_n(\oef)$. Let the ring
\[
W_n(\oef/\beF)/([\zeta_{p^n}]-1)^r W_n(m_F/\beF)
\]
be denoted by $\bar{A}_{n,F,r+}$.
We also put $\bar{A}_{n,r+}=\bar{A}_{n,\kbar,r+}$.

\begin{lem}\label{Wittlem}
The ideal $([\zeta_{p^n}]-1)^rW_n(m_F)$ of $W_n(\oef)$ contains the ideal $W_n(\beF)$ for any $r\in\{0,\ldots,p-2\}$. We also have $(([\zeta_{p^n}]-1)^{p-1})\supseteq W_n(p\oef)$. 
\end{lem}
\begin{proof}
The proof is similar to the proof of Lemma \ref{ker}. Let us show the first assertion. Since this is trivial for $r=0$, we may assume $r\geq 1$. Put $x=(x_0,\ldots,x_{n-1})=([\zeta_{p^n}]-1)^r\in W_n(\oef)$. Then we have $v_p(x_0)=r/(p^{n-1}(p-1))$. By induction, it is enough to show that for $0\leq i\leq n-1$, if $(x_0,\ldots,x_i)(0,\ldots,0,y_i)\in W_{i+1}(\beF)$, then $y_i\in m_F$ and $x(0,\ldots,0,y_i,0,\ldots,0)\in W_n(\beF)$. By assumption, we have
\[
v_p(y_i)>\frac{r}{p-1}(1-\frac{1}{p^{n-i-1}})\geq 0.
\]
Put $(0,\ldots,0,w_i,\ldots,w_{n-1})=x(0,\ldots,0,y_i,0,\ldots,0)$. We show $w_l\in \beF$ for any $l$ by induction. Indeed, let us suppose that $w_l\in\beF$ for any $i\leq l \leq k-1$ with some $i+1\leq k\leq n-1$. We have the equality
\[
p^iy_{i}^{p^{k-i}}(x_{0}^{p^k}+px_{1}^{p^{k-1}}+\cdots+p^kx_k)=(p^iw_{i}^{p^{k-i}}+p^{i+1}w_{i+1}^{p^{k-i-1}}+\cdots+p^kw_k).
\] 
Since $r\geq 1$, we have $(p^{k-l}-1)r/(p-1)\geq k-l$ for $0\leq l\leq k-1$. This implies $v_p(p^lw_{l}^{p^{k-l}})> k+r/(p-1)$ for $0\leq l\leq k-1$. The valuation of the left-hand side of the above equality also satisfies this inequality. Thus we have $v_p(w_k)>r/(p-1)$ and the assertion follows. We can show the second assertion similarly.
\end{proof}

By this lemma, the natural surjections of rings
\begin{align*}
W_n(\oef)/&([\zeta_{p^n}]-1)^rW_n(m_F)\\
&\to W_n(\oef/p\oef)/([\zeta_{p^n}]-1)^rW_n(m_F/p\oef)\to \bar{A}_{n,F,r+}
\end{align*}
are isomorphisms. Then we see that the natural injection $F\to \kbar$ induces an injection of rings $\bar{A}_{n,F,r+}\to \bar{A}_{n,r+}$.

Write $Z_n$ for the image of the
element $Z$ of $\Acrys$ in $W_{n}^{\PD}(\tokbar)$. Then we have a commutative
diagram of $\Sigma$-algebras
\[
\xymatrix{
& \hat{A}_n \ar@{->>}[dd]\ar[r] & \Acrys/p^n\Acrys\ar[d]^{\wr} \\
& &W_{n}^{\PD}(\tokbar)\ar@{->>}[dd]\\
W_n(R)/(([\uep]-1)^{p-1})\ar[r]^-{\sim}\ar[d]_{\wr}& \hat{A}_n/(Z)\ar[rd] & \\
\bar{A}_{n,p-1}\ar[rr]\ar@{->>}[d] & &W_{n}^{\PD}(\tokbar)/(Z_n),\\
\bar{A}_{n,r+} & &
}
\]
where all the vertical arrows are surjections satisfying the condition of
Corollary \ref{quotlem}. Hence this is also a commutative diagram
in $'\Mod^{r,\phi}_{/\Sigma}$. Note that these rings and homomorphisms are independent of the
choice of a system $\{\zeta_{p^n}\}_{n\in\mathbb{Z}_{\geq 0}}$. We also note that $\Fil^r\bar{A}_{n,r+}=E([\pi_n])^r\bar{A}_{n,r+}$ and $\phi_r(E([\pi_n])^ry)=c^r\phi(y)$ for any $y\in\bar{A}_{n,r+}$, where $\phi$ denotes the Frobenius endomorphism of $\bar{A}_{n,r+}$ induced from that of the ring $W_n(\okbar/\bekbar)$. 
Moreover, let $M$ be an object of $\Mod^{r,\phi}_{/\Sigma_\infty}$. Then,
by Corollary \ref{quotlem} and
Corollary \ref{AnAcrys}, we have a natural
isomorphism of abelian groups
\[
\Hom_{\Sigma,\Fil^r,\phi_r}(M,W_{n}^{\PD}(\tokbar))\to\Hom_{\Sigma,\Fil^r,\phi_r}(M,\bar{A}_{n,r+}).
\]

Next we investigate the module on the right-hand side of this isomorphism, and prove this is in fact an isomorphism of $G_{F_n}$-modules. Consider the element $E([\pi_n])\in W_n(\oefn/p\oefn)$ and let us fix its lift $\hat{\gamma}\in W_n(\mathcal{O}_{F_n})$ by the natural surjection $W_n(\oefn)\to W_n(\oefn/p\oefn)$. Let $a\in W(R)^\times$ and $v=t/E([\upi])\in W(R)^\times$ as before. We let $a_n$, $t_n$ and $v_n$ denote the images of $a$, $t$ and $v$ by the surjection $W(R)\to W_n(\tokbar)$ induced by $\beta_n$, respectively. The elements $a_n$ and $t_n$ of the ring $W_n(\tokbar)$ are contained in the subring $W_n(\mathcal{O}_{F_n}/p\mathcal{O}_{F_n})$. We abusively let them also denote their images by the natural surjections $W_n(\tokbar)\to W_n(\okbar/\bekbar)\to \bar{A}_{n,r+}$. 

\begin{lem}\label{divgamma}
The element 
\[
\hat{t}_n=1+[\zeta_{p^{n+1}}]+[\zeta_{p^{n+1}}]^{2}+\cdots+[\zeta_{p^{n+1}}]^{p-1}=\frac{[\zeta_{p^n}]-1}{[\zeta_{p^{n+1}}]-1}
\]
is divisible by $\hat{\gamma}$ in the ring $W_n(\oefn)$. In particular, $\hat{\gamma}$ is a non-zero divisor of the ring $W_n(\okbar)$.
\end{lem}
\begin{proof}
It is enough to show the divisibility in the ring $W_n(\okbar)$.
Note that the element $t_n$ is also the image of $\hat{t}_n$ by the natural map $W_n(\oefn)\to W_n(\oefn/p\oefn)$. 
Let $\hat{v}_n$ be a lift of $v_n$ by the natural surjection $W_n(\okbar)\to W_n(\tokbar)$. Then we have $\hat{t}_n-\hat{\gamma}\hat{v}_n\in W_n(p\okbar)$. By Lemma \ref{Wittlem}, there exists $\hat{y}\in W_n(m_{\kbar})$ such that $\hat{t}_n-\hat{\gamma}\hat{v}_n=\hat{t}_n\hat{y}$. Hence we have $\hat{t}_n(1-\hat{y})=\hat{\gamma}\hat{v}_n$. Since $\hat{y}$ is topologically nilpotent in the ring $W_n(\okbar)$, the element $1-\hat{y}$ is invertible and the lemma follows.
\end{proof}

\begin{lem}\label{Fn}
The image of $Y\in\Sigma$ in the ring
 $\bar{A}_{n,r+}$ $(${\it resp.} $\bar{A}_{n,p-1}${}$)$ is contained in its subring
 $\bar{A}_{n,F_n,r+}$ $(${\it resp.} $W_n(\oefn/p\oefn)/(([\zeta_{p^n}]-1)^{p-1})${}$)$. 
\end{lem}
\begin{proof}
We have the equality
\[
E([\pi_n])v_n=t_n=1+[\zeta_{p^{n+1}}]+[\zeta_{p^{n+1}}]^2+\cdots+[\zeta_{p^{n+1}}]^{p-1}
\]
in the ring $W_n(\tokbar)$.
Note that any element
 $v'_n\in W_n(\tokbar)$ satisfying the same equality is
 invertible and thus the elements $(v'_n)^{-1}E([\pi_n])$ are equal
 to each other. Since $Y=-a_nv_{n}^{-1}E([\pi_n])^{p-1}$ in
the rings $\bar{A}_{n,r+}$ and $\bar{A}_{n,p-1}$, it suffices to construct an element
 $v'_n$ of the ring $W_n(\mathcal{O}_{F_n}/p\oefn)$ such
 that the equality $E([\pi_n])v'_n=t_n$ holds. This follows from Lemma \ref{divgamma}.
\end{proof}

From this lemma, we see that the natural $G_{F_n}$-actions on the rings
$\bar{A}_{n,p-1}$ and $\bar{A}_{n,r+}$ are compatible with the filtered $\phi_r$-module
structures over $\Sigma$. In the big commutative diagram above, the lowest
horizontal arrow and lower right vertical arrow are $G_{K}$-linear by
definition. Hence we have shown the following proposition.

\begin{prop}\label{comparison}
Let $M$ be an object of $\Mod^{r,\phi}_{/\Sigma_\infty}$. Then the map
\[
\Hom_{\Sigma,\Fil^r,\phi_r}(M,W_{n}^{\PD}(\tokbar))\to \Hom_{\Sigma,\Fil^r,\phi_r}(M,\bar{A}_{n,r+})
\]
is an isomorphism of $G_{F_n}$-modules.
\end{prop}

Let $M$ be as in the proposition. Let $e_1,\ldots,e_d$ be a system of generators of $M$ as in Lemma \ref{adaptedlem} and 
$C=(c_{i,j})\in M_d(\Sigma)$ be a matrix representing
$\phi_r$ as in Corollary \ref{C}. Consider the surjection $\Sigma^{\oplus d}\to M$ defined by $(s_1,\ldots,s_d)\mapsto s_1e_1+\cdots+s_de_d$ and let $(s_{1,1},\ldots,s_{1,d}),\ldots,(s_{q,1},\ldots,s_{q,d})$ be a system of generators of its kernel. Then the underlying $G_{F_n}$-set of the $G_{F_n}$-module
\[
\Hom_{\Sigma,\Fil^r,\phi_r}(M,\bar{A}_{n,r+})
\] is
identified with the set of $d$-tuples $(\bar{x}_1,\ldots,\bar{x}_d)$ in
$\bar{A}_{n,r+}$ such that the following three conditions hold:
\begin{itemize}
\item 
$s_{l,1}\bar{x}_1+\cdots+s_{l,d}\bar{x}_d=0$ for any $l$,
\item
$c_{1,i}\bar{x}_1+\cdots+c_{d,i}\bar{x}_d\in\Fil^r\bar{A}_{n,r+}$ for any $i$,
\item the following
equality holds:
\begin{equation*}
\left\{
\begin{array}{c}
\phi_r(c_{1,1}\bar{x}_1+\cdots+c_{d,1}\bar{x}_d)=\bar{x}_1 \\
\vdots\\
\phi_r(c_{1,d}\bar{x}_1+\cdots+c_{d,d}\bar{x}_d)=\bar{x}_d.
\end{array}
\right.
\end{equation*}
\end{itemize}

We choose lifts
$\hat{c}$, $\hat{c}_{i,j}$ and $\hat{s}_{i,j}$ in $W_n(\mathcal{O}_{F_n})$ of the images of $c$, $c_{i,j}$ and $s_{i,j}$ in
$\bar{A}_{n,r+}$ by the natural ring homomorphism 
\[
W_n(\okbar)\to W_n(\tokbar)\to W_n(\okbar/\bekbar) \to \bar{A}_{n,r+},
\]
respectively. Recall that we have already chosen a lift $\hat{\gamma}\in W_n(\mathcal{O}_{F_n})$ of $E([\pi_n])\in W_n(\oefn/p\oefn)$. 

Fix a polynomial $\Phi_i \in \mathbb{Z}[X_0,\ldots,X_{n-1}]$ such that $\Phi_i\equiv X_{i}^{p}\text{ mod }p$. This induces for any commutative ring $B$ a map $\Phi=(\Phi_0,\ldots,\Phi_{n-1}):W_n(B)\to W_n(B)$ which is a lift of
the Frobenius endomorphism on $W_n(B/pB)$. 
In particular, set
$B$ to be the polynomial ring $\mathbb{Z}[X_0,\ldots,X_{n-1},Y_0,\ldots,Y_{n-1}]$. Put
$X=(X_0,\ldots,X_{n-1})$ and $Y=(Y_0,\ldots,Y_{n-1})$ in the ring
$W_n(B)$. Then we see that there exists
elements $U_0,\ldots,U_{n-1}$ and $U'_0,\ldots,U'_{n-1}$ of the
polynomial ring $B$
such that
\begin{align*}
\Phi(X+Y)&=\Phi(X)+\Phi(Y)+(pU_0,\ldots,pU_{n-1}),\\
\Phi(XY)&=\Phi(X)\Phi(Y)+(pU'_0,\ldots,pU'_{n-1})
\end{align*}
in the ring $W_n(B)$.

\begin{prop}\label{liftinglem}

Every $d$-tuple $(\bar{x}_1,\ldots,\bar{x}_d)$ in $\bar{A}_{n,r+}$ satisfying the above three conditions 
uniquely lifts to
 a $d$-tuple $(\hat{x}_1,\ldots,\hat{x}_d)$ in $W_n(\okbar)$ such that
\begin{itemize}
\item $\hat{s}_{l,1}\hat{x}_1+\cdots+\hat{s}_{l,d}\hat{x}_d\in ([\zeta_{p^n}]-1)^rW_n(m_{\kbar})$ for any $l$,
\item $\hat{c}_{1,i}\hat{x}_1+\cdots+\hat{c}_{d,i}\hat{x}_d\in
 \hat{\gamma}^rW_n(\okbar)$ for any $i$,
\item the following equality holds:
\begin{equation*}
\left\{
\begin{array}{c}
\hat{c}^r\Phi((\hat{c}_{1,1}\hat{x}_1+\cdots+\hat{c}_{d,1}\hat{x}_d)/\hat{\gamma}^r)=\hat{x}_1 \\
\vdots\\
\hat{c}^r\Phi((\hat{c}_{1,d}\hat{x}_1+\cdots+\hat{c}_{d,d}\hat{x}_d)/\hat{\gamma}^r)=\hat{x}_d.
\end{array}
\right.
\end{equation*}
\end{itemize}
\end{prop}

\begin{proof}

Fix a lift $\hat{x}_i$ of $\bar{x}_i$ in $W_n(\okbar)$. Recall that the kernel of the surjection $W_n(\okbar)\to \bar{A}_{n,r+}$ is equal to the ideal $([\zeta_{p^n}]-1)^rW_n(m_{\kbar})$. The first condition in the proposition holds automatically for $(\hat{x}_1,\ldots,\hat{x}_d)$. By Lemma \ref{divgamma}, the element $\hat{c}_{1,i}\hat{x}_1+\cdots+\hat{c}_{d,i}\hat{x}_d$ is contained in $\hat{\gamma}^rW_n(\okbar)$ for any $i$. Since the map $\phi_r:\Fil^r\bar{A}_{n,r+}\to \bar{A}_{n,r+}$ satisfies $\phi_r(E([\pi_n])^r\bar{x})=c^r\phi(\bar{x})$ for any $\bar{x}\in\bar{A}_{n,r+}$, we have
\[
\left\{
\begin{array}{c}
\hat{c}^r\Phi((\hat{c}_{1,1}\hat{x}_1+\cdots+\hat{c}_{d,1}\hat{x}_d)/\hat{\gamma}^r)=\hat{x}_1+([\zeta_{p^n}]-1)^r\hat{\delta}_1 \\
\vdots\\
\hat{c}^r\Phi((\hat{c}_{1,d}\hat{x}_1+\cdots+\hat{c}_{d,d}\hat{x}_d)/\hat{\gamma}^r)=\hat{x}_d+([\zeta_{p^n}]-1)^r\hat{\delta}_d
\end{array}
\right.
\]
for some $\hat{\delta}_1,\ldots,\hat{\delta}_d\in W_n(m_{\kbar})$. It suffices to show that there
 exists a unique $d$-tuple $(\hat{y}_1,\ldots,\hat{y}_d)$ in
 $W_n(m_{\kbar})$ such that
\begin{align*}
\hat{c}^r\Phi((\hat{c}_{1,i}(\hat{x}_1+([\zeta_{p^n}]-1)^{r}\hat{y}_1)&+\cdots+\hat{c}_{d,i}(\hat{x}_d+([\zeta_{p^n}]-1)^{r}\hat{y}_d))/\hat{\gamma}^r)\\
&=\hat{x}_i+([\zeta_{p^n}]-1)^{r}\hat{y}_i
\end{align*}
for any $i$. For this, we need the following lemma.

\begin{lem}\label{recursion}
Let $N$ be a complete discrete valuation field and $m_N$ be the maximal ideal of
 $N$. Let $\epsilon_1,\ldots,\epsilon_d$ be
 in $m_N$. Let $P_1,\ldots,P_d$ and $P'_1\ldots,P'_d$ be
 elements of $\mathcal{O}_N[[Y_1,\ldots,Y_d]]$ such that 
$P_i\in (Y_1,\ldots,Y_d)^2$. Then the equation
\[
\left\{
\begin{array}{c}
Y_1-P_1(Y_1,\ldots,Y_d)- \epsilon_1P'_1(Y_1,\ldots,Y_d)=0\\
\vdots\\
Y_d-P_d(Y_1,\ldots,Y_d)- \epsilon_dP'_d(Y_1,\ldots,Y_d)=0
\end{array}
\right.
\]
has a unique solution in $m_N$.
\end{lem}
\begin{proof}
By assumption, we see that for any integer $l\geq 1$, a $d$-tuple
 $(y_1,\ldots,y_d)$ in $m_N/m_{N}^{l}$ satisfying the above equation
 lifts uniquely to a $d$-tuple in $m_N/m_{N}^{l+1}$ satisfying the same
 equation. Thus the lemma follows.
\end{proof}

Let us write as $\hat{y}_i=(\hat{y}_{i,0},\ldots,\hat{y}_{i,n-1})$. Since the image of $\Phi(([\zeta_{p^{n+1}}]-1)^r)$
 in $\bar{A}_{n,r+}$ is equal to $([\zeta_{p^n}]-1)^r$, we can find $\hat{b}\in W_n(\okbar)$
 such that
\[
\Phi(([\zeta_{p^n}]-1)^r/\hat{\gamma}^r)=([\zeta_{p^n}]-1)^r \hat{b}.
\]
Then there exists polynomials $U_{i,m}$ over $\okbar$ of the indeterminates 
$\underline{Y}=(Y_{i,m})_{1\leq i\leq d,0\leq m\leq n-1}$ such that the
 equation we have to solve is
\begin{align*}
\hat{x}_i+([\zeta_{p^n}]-1)^r\hat{y}_i&=\hat{x}_i+([\zeta_{p^n}]-1)^r\hat{\delta}_i\\
&\qquad
 +([\zeta_{p^n}]-1)^r\hat{b}\hat{c}^r(\Phi(\hat{c}_{1,i})\Phi(\hat{y}_1)+\cdots+\Phi(\hat{c}_{d,i})\Phi(\hat{y}_d))\\
&\qquad +(pU_{i,0}(\underline{\hat{y}}),\ldots,pU_{i,n-1}(\underline{\hat{y}}))
\end{align*}
for any $i$, where we put $\underline{\hat{y}}=(\hat{y}_{i,m})_{1\leq i\leq d,0\leq m\leq n-1}$. As in the proof of Lemma \ref{Wittlem}, we see that, for any elements
$P_0,\ldots,P_{n-1}$ of the polynomial ring $\okbar[\underline{Y}]$,
we can uniquely find elements $Q_0,\ldots,Q_{n-1}$ of this ring such that the
coefficients of these polynomials are in the maximal ideal $m_{\kbar}$
 and the equality
\[
(pP_0,\ldots,pP_{n-1})=([\zeta_{p^n}]-1)^r(Q_0,\ldots,Q_{n-1})
\]
holds in the ring of Witt vectors $W_n(\okbar[\underline{Y}])$. Therefore, this equation is equivalent to the equation
\begin{align*}
\hat{y}_i&=\hat{\delta}_i+\hat{b}\hat{c}^r(\Phi(\hat{c}_{1,i})\Phi(\hat{y}_1)+\cdots+\Phi(\hat{c}_{d,i})\Phi(\hat{y}_d))\\
&\qquad +(V_{i,0}(\underline{\hat{y}}),\ldots,V_{i,n-1}(\underline{\hat{y}})),
\end{align*}
where $V_{i,m}$ is a polynomial of $\underline{Y}$ over $\okbar$ whose
 coefficients are in the maximal ideal $m_{\kbar}$.
From the definition of $\Phi$, we see that $\underline{\hat{y}}=(\hat{y}_{i,m})_{i,m}$ is a
 solution of a system of equations
\[
Y_{i,m}-P_{i,m}(\underline{Y})-\epsilon_{i,m}P'_{i,m}(\underline{Y})=0
\]
satisfying the
 condition of Lemma \ref{recursion} for a sufficiently large
 finite extension $N$ of $K$. Then, by this lemma, we can solve the
 equation uniquely in $m_{\kbar}$.
\end{proof}

Let $F$ be an algebraic extension of $F_n$ in $\kbar$ and
consider the ring $\bar{A}_{n,F,r+}$.
By Lemma \ref{Fn}, we can consider this ring as a $\Sigma$-subalgebra of $\bar{A}_{n,r+}$. Put $\Fil^r\bar{A}_{n,F,r+}=E([\pi_n])^r\bar{A}_{n,F,r+}$. Then Lemma \ref{divgamma} implies that 
\[
\bar{A}_{n,F,r+}\cap\Fil^r\bar{A}_{n,r+}=\Fil^r\bar{A}_{n,F,r+}.
\]
Moreover, the Frobenius endomorphism $\phi$ of the ring $\bar{A}_{n,r+}$ preserves the subalgebra $\bar{A}_{n,F,r+}$ and thus $\phi_r:\Fil^r\bar{A}_{n,r+}\to\bar{A}_{n,r+}$ induces a $\phi$-semilinear map $\phi_r:\Fil^r\bar{A}_{n,F,r+}\to\bar{A}_{n,F,r+}$. Hence $\bar{A}_{n,F,r+}$ is a subobject of $\bar{A}_{n,r+}$ in the category $'\Mod^{r,\phi}_{/\Sigma}$. For
$M\in\Mod^{r,\phi}_{/\Sigma_\infty}$, let us set
\[
T_{\mathrm{crys},\pi_n,F}^{*}(M)=\Hom_{\Sigma,\Fil^r,\phi_r}(M,\bar{A}_{n,F,r+}).
\]
We see that 
\[
\bar{A}_{n,r+}=\bar{A}_{n,\kbar,r+}=\bigcup_{F/F_n} \bar{A}_{n,F,r+}
\]
in $'\Mod^{r,\phi}_{/\Sigma}$ and thus we have a natural identification
of abelian groups
\[
T_{\mathrm{crys},\pi_n,\kbar}^{*}(M)=\bigcup_{F/F_n}T_{\mathrm{crys},\pi_n,F}^{*}(M).
\]
The absolute Galois group $G_{F_n}$ acts on the abelian group on the left-hand side.

\begin{lem}\label{fix}

Let $F$ be an algebraic extension of $F_n$ in $\kbar$. Then the $G_F$-fixed part $T_{\mathrm{crys},\pi_n,\kbar}^{*}(M)^{G_F}$ is equal to
$T_{\mathrm{crys},\pi_n,F}^{*}(M)$.
\end{lem}
\begin{proof}
From Proposition \ref{liftinglem}, we see that the elements of $T_{\mathrm{crys},\pi_n,\kbar}^{*}(M)$
 correspond bijectively to the $d$-tuples in $W_n(\okbar)$ satisfying the three conditions in this proposition. The uniqueness assertion of the proposition shows that $g\in G_{F}$ fixes such a $d$-tuple in $W_n(\okbar)$ if and only if $g$ fixes its image in
 $\bar{A}_{n,r+}$. Hence an element of $T_{\mathrm{crys},\pi_n,\kbar}^{*}(M)$ is fixed by $G_F$ if and only if it is contained in the image of $W_n(\oef)$. Thus the lemma follows. 
\end{proof}

\begin{cor}\label{Fon}
Let $L_n$ be the finite Galois extension of $F_n$ corresponding to the
 kernel of the map 
\[
G_{F_n}\to\Aut(T_{\mathrm{crys},\pi_n,\kbar}^{*}(M)).
\]
Then an algebraic
 extension $F$ of $F_n$ in $\kbar$ contains $L_n$ if and
 only if 
\[
\# T_{\mathrm{crys},\pi_n,F}^{*}(M)=\# T_{\mathrm{crys},\pi_n,\kbar}^{*}(M).
\]
\end{cor}
\begin{proof}
An algebraic extension $F$ of $F_n$ contains $L_n$ if and only if the
 action of $G_F$ on $T_{\mathrm{crys},\pi_n,\kbar}^{*}(M)$ is trivial. By Lemma \ref{fix}, this
 is equivalent to $T_{\mathrm{crys},\pi_n,F}^{*}(M)=T_{\mathrm{crys},\pi_n,\kbar}^{*}(M)$. 
\end{proof}



\section{Ramification bound}

In this section, we prove Theorem \ref{main_br}. Let $\cM$ be an object of $\Mod^{r,\phi,N}_{/S_\infty}$ which is killed by $p^n$ and let $L$ be the finite
Galois extension of $K$
corresponding to the kernel of the map 
\[
G_K\to\Aut(\Tsta(\cM)).
\]
Then the theorem is equivalent to the inequality $u_{L/K}\leq u(K,r,n)$, where $u_{L/K}$ denotes the greatest upper ramification break of the Galois extension $L/K$ (\cite{F}). For $r=0$, the $G_K$-module $\Tsta(\cM)$ is unramified and the assertion is trivial. Thus we may assume $p\geq 3$ and $r\geq 1$.

Let $L_n$ be the finite
Galois extension of $F_n$
corresponding to the kernel of the map 
\[
G_{F_n}\to\Aut(\Tsta(\cM)).
\]
Since $F_n$ is Galois over $K$, the extension $L_n=LF_n$ is also a Galois extension of $K$. Let $M\in\Mod^{r,\phi}_{/\Sigma_\infty}$ be the filtered $\phi_r$-module over $\Sigma$ which corresponds to $\cM$ by the equivalence $\cM_{\Sigma_\infty}$ of Proposition \ref{SSigeq}. Then Proposition \ref{vsTcrys} and Proposition \ref{comparison} show that $L_n$ is also the finite extension of $F_n$ cut
out by the $G_{F_n}$-module $T^{*}_{\mathrm{crys},\pi_n,\kbar}(M)$. It is enough to prove the inequality $u_{L_n/K}\leq u(K,r,n)$.

Before proving this, we state some general lemmas to calculate the
ramification bound. Let $N$ be a complete discrete valuation field of positive residue characteristic, $v_N$ be its valuation
 normalized as $v_N(N^\times)=\mathbb{Z}$
 and $N^{\mathrm{sep}}$ be its separable closure. We extend $v_N$ to any algebraic closure of $N$.

\begin{lem}\label{AbbesSaito}
Let $f(T)\in\oen[T]$
 be a separable monic polynomial and $z_1,\ldots,z_d$ be the
 zeros of $f$ in $\mathcal{O}_{N^{\mathrm{sep}}}$. Suppose that the set
 $\{v_N(z_k-z_i)\mid k=1,\ldots,d, k\neq i\}$ is
 independent of $i$. Put
\[
s_f=\sum_{\substack{
k=1,\ldots,d\\
k\neq i}}
v_N(z_k-z_i)\ \text{and}\
 \alpha_f=\sup_{\substack{
k=1,\ldots,d\\
k\neq i}}
v_N(z_k-z_i),
\]
which are independent of $i$ by assumption.
If $j>s_f+\alpha_f$, then we have the decomposition
\[
\{x\in\mathcal{O}_{N^{\mathrm{sep}}}\mid v_N(f(x))\geq j\}=\coprod_{i=1,\ldots,d}\{x\in\mathcal{O}_{N^{\mathrm{sep}}}\mid v_N(x-z_i)\geq j-s_f\}.
\]
Otherwise, the set on the left-hand side contains
\[
\{x\in\mathcal{O}_{N^{\mathrm{sep}}}\mid v_N(x-z_i)\geq \alpha_f\},
\]
which contains at least two zeros of $f$.
\end{lem}
\begin{proof}
A verbatim argument in the proof of \cite[Lemma 6.6]{AS1} shows the claim.
\end{proof}

\begin{cor}\label{ramCI}
Let $f(T)$ be as above and put $B=\oen[T]/(f(T))$. Let us
 write the $N$-algebra $N'=B\otimes_{\oen}N$ as the
 product $N_1\times\cdots\times N_t$ of finite separable extensions
 $N_1,\ldots,N_t$ of $N$. If $j>s_f+\alpha_f$, then the $j$-th upper
 numbering ramification group $($\cite{AS1}$)$, which we let be denoted by $G^{(j)}_{N}$, is contained in $G_{N_i}$ for any $i$. Moreover, if $N'$
 is a field and $B$ coincides with $\mathcal{O}_{N'}$, then $j>s_f+\alpha_f$
 if and only if $G_{N}^{(j)}\subseteq G_{N'}$.
\end{cor}
\begin{proof}
Note that the algebra
 $B$ is finite flat and of relative complete intersection over
 $\oen$. By the previous lemma, the conductor $c(B)$ of the $\oen$-algebra
 $B$ (\cite[Proposition 6.4]{AS1}) is equal to
 $s_f+\alpha_f$. Thus we have the inequality
\[
c(\mathcal{O}_{N_1}\times\cdots\times\mathcal{O}_{N_t})\leq c(B)=s_f+\alpha_f
\]
by the definition of the conductor and a functoriality of the functor
 $\mathcal{F}^j$ defined in \cite{AS1}. This implies the corollary.
\end{proof}

\begin{cor}\label{ramlemz}
We have the inequality
\[
u_{K(\zeta_{p^{n+1}})/K}\leq 1-\frac{1}{e(K(\zeta_p)/K)}+e(n+\frac{1}{p-1}),
\]
where $e(K(\zeta_p)/K)$ denotes the relative ramification index of $K(\zeta_p)$ over $K$.
\end{cor}
\begin{proof}
Since the Herbrand function is transitive and the finite extension $K(\zeta_p)$ is tamely ramified over $K$, it is enough to show the inequality
\[
u_{K(\zeta_{p^{n+1}})/K(\zeta_p)}\leq e(K(\zeta_p))(n+\frac{1}{p-1}).
\]
Put $N=K(\zeta_p)$ and $f(T)=T^{p^n}-\zeta_p$. These satisfy the assumptions of Corollary \ref{ramCI}. We have $s_f=ne(K(\zeta_p))$ and $\alpha_f=e(K(\zeta_p))/(p-1)$ in this case. Hence the corollary follows.
\end{proof}

\begin{cor}\label{ramlem}
Consider the finite Galois extension $F_n=K_n(\zeta_{p^{n+1}})$ of $K$. Then we have the equality
\[
u_{F_n/K}=1+e(n+\frac{1}{p-1}).
\] 
\end{cor}
\begin{proof}
Applying
 Corollary \ref{ramCI} to the Eisenstein polynomial $f(T)=T^{p^n}-\pi$ and $N=K$
 shows that $j>1+e(n+1/(p-1))$ if and only if $G_{K}^{(j)}\subseteq G_{K_n}$. From Corollary \ref{ramlemz}, we see that if $j>1+e(n+1/(p-1))$, then $G_{K}^{(j)}\subseteq G_{K(\zeta_{p^{n+1}})}$. Since 
$G_{F_n}=G_{K_n}\cap G_{K(\zeta_{p^{n+1}})}$, we conclude that
 $j>1+e(n+1/(p-1))$ if and only if $G_{K}^{(j)}\subseteq G_{F_n}$.
\end{proof}

\begin{rmk}\label{Tate}
Note that this argument also shows the equality
\[
u_{K_n(\zeta_{p^n})/K}=1+e(n+\frac{1}{p-1}).
\]
\end{rmk}

Next we assume that the residue field of $N$ is perfect. For an algebraic extension $F$ of
$N$, we put
\[
\mathfrak{a}_{F/N}^{j}=\{x\in\oef\mid v_N(x)\geq j\}.
\] 
Let $Q$ be a finite Galois extension of $N$ and consider the property
\[
(P_j)
\left\{
\begin{array}{l}
\text{for any algebraic extension $F$ of $N$, if there exists}\\
\text{an
 $\oen$-algebra homomorphism $\mathcal{O}_{Q}\to\oef/\mathfrak{a}_{F/N}^{j}$,}\\
\text{then there
 exists an $N$-algebra injection $Q\to F$}
\end{array}
\right.
\]
for $j\in\mathbb{R}_{\geq 0}$, as in \cite[Proposition 1.5]{F}.
Then we have the following proposition, which is due to Yoshida. Here we reproduce his proof for the convenience
of the reader.

\begin{prop}[\cite{Yos}]\label{Yos}
\[
u_{Q/N}=\inf\{j\in\mathbb{R}_{\geq 0}\mid \text{the property }(P_j)\text{ holds}\}.
\]
\end{prop}
\begin{proof}
By \cite[Proposition 1.5 (i)]{F}, it is enough to show that the property $(P_j)$ does not hold for
 $j=u_{Q/N}-(e')^{-1}$ with an arbitrarily large $e'>0$. As in the proof
 of \cite[Proposition 1.5 (ii)]{F}, we may assume that $Q$ is totally and wildly ramified over $N$. Take an arbitrarily large integer $e''>0$ with $(e'',pe(Q/N))=1$. We may also assume that $N$ contains a primitive $e''$-th root of unity. Set $N'=N(\pi_{N}^{1/e''})$ and $Q'=QN'$. Note that we have $u_{Q'/N}=u_{Q/N}$ by assumption. From this
 proposition in \cite{F}, we see that for some algebraic extension $F$ of $N$, there exists an
 $\oen$-algebra homomorphism
 $\mathcal{O}_{Q'}\to\oef/\mathfrak{a}_{F/N}^{j}$ for
 $j=u_{Q/N}-e(Q'/N)^{-1}$ but no $N$-algebra injection $Q'\to F$. Since
 $Q/N$ is wildly ramified, we see that $e(Q/N)u_{Q/N}-1>e(Q/N)$. Hence
 we have $u_{Q/N}-e(Q'/N)^{-1}>1\geq u_{N'/N}$ and there exists an
 $N$-algebra injection $N'\to F$ also by this proposition. Thus there exists no $N$-algebra
 injection $Q\to F$ and the property $(P_j)$ for $Q/N$ does not
 hold. Since $e(Q'/N)=e''e(Q/N)$, the proposition follows.
\end{proof}

We see from Proposition \ref{Yos} that to bound the greatest upper
ramification break $u_{L_n/K}$, it is enough to show the following
proposition.
 
\begin{prop}\label{Pm'}
Let $F$ be an algebraic extension of $K$. If
 $j>u(K,r,n)$ and there exists an $\okey$-algebra homomorphism 
\[
\eta:\mathcal{O}_{L_n}\to\oef/\mathfrak{a}_{F/K}^{j},
\]
then there exists a $K$-algebra injection $L_n\to F$.
\end{prop}
\begin{proof}
We may assume that $F$ is contained in $\kbar$. By assumption, we
 have $j > er/(p-1)$ and we see that
the ideal $\beF=\{x\in\oef\mid v_K(x)> er/(p-1)\}$ contains $\aFj$. Thus $\eta$ induces an $\okey$-algebra homomorphism
\[
\mathcal{O}_{L_n}\to \oef/\beF.
\]

Since $\eta$ also induces an
 $\okey$-algebra homomorphism 
$\mathcal{O}_{F_n}\to \oef/\mathfrak{a}_{F/K}^{j}$ and $r\geq 1$, from Corollary
 \ref{ramlem} and \cite[Proposition
 1.5]{F} we get a $K$-linear injection $F_n\to F$. Thus we see that $F$ contains
 $\pi_n$ and $\zeta_{p^{n+1}}$. More precisely, we have the following two lemmas.
 
\begin{lem}\label{prei}
There exists $i\in\bZ$ such that $\eta(\pi_n)\equiv \pi_n\zeta_{p^n}^{i}\text{ $\mathrm{mod}$ }\beF$.
\end{lem}
\begin{proof}
Since the map $\eta$ is $\okey$-linear, the equality $\eta(\pi_n)^{p^n}=\pi$ holds in
 $\oef/\aFj$. Set $\hat{x}$ to be a lift of $\eta(\pi_n)$ in
 $\oef$. Then we have
\[
v_K(\hat{x}^{p^n}-\pi)=\sum_{i=0}^{p^n-1}v_K(\hat{x}-\pi_n\zeta_{p^n}^{i})\geq j.
\]
Let us apply Lemma \ref{AbbesSaito} to $f(T)=T^{p^n}-\pi\in\okey[T]$. Then, with the
 notation of the lemma, we have
\[
s_f=1-\frac{1}{p^n}+ne\ \text{and}\ \alpha_f=\frac{1}{p^n}+\frac{e}{p-1}.
\]
Since $j-s_f>er/(p-1)$ by assumption, we have 
\[
\hat{x} \equiv \pi_n\zeta_{p^n}^{i}\text{ mod }\beF
\]
for some $i$. 
\end{proof}
 
\begin{lem}\label{z}
There exists $g'\in G_K$ such that $\eta(\zeta_{p^{n+1}})\equiv g'(\zeta_{p^{n+1}})\text{ $\mathrm{mod}$
 }\beF$. 
\end{lem}
\begin{proof}
Set $N$ to be the maximal unramified subextension of $K(\zeta_{p^{n+1}})/K$. Since the map $\okey\to\mathcal{O}_{N}$ is etale, there exists a $K$-algebra injection $g_0:N\to F$ such that $\eta(x)\equiv g_0(x) \text{ $\mathrm{mod}$ }\mathfrak{a}_{F/K}^{j}$ for any $x\in\mathcal{O}_{N}$. Let $\varpi$ be a uniformizer of $K(\zeta_{p^{n+1}})$ and $f(T)\in \mathcal{O}_{N}[T]$ be the Eisenstein polynomial of $\varpi$ over $\mathcal{O}_{N}$. We let $f^{g_0}(T)\in\mathcal{O}_{N}[T]$ denote the conjugate of $f$ by $g_0$. Then $f^{g_0}$ satisfies the conditions of Lemma \ref{AbbesSaito}. By definition we have $s_{f^{g_0}}=s_f$ and $\alpha_{f^{g_0}}=\alpha_f$. Since the roots of $f^{g_0}(T)$ are conjugates of $\varpi$ over $K$, Lemma \ref{AbbesSaito} implies as in the previous lemma that there exists $g'\in G_K$ such that $g'|_N=g_0$ and $\eta(\varpi)\equiv g'(\varpi) \text{ $\mathrm{mod}$ }\mathfrak{a}_{F/K}^{j-s_f}$. Since $\mathcal{O}_{K(\zeta_{p^{n+1}})}$ is generated by $\varpi$ over $\oen$, we see that $\eta(\zeta_{p^{n+1}})\equiv g'(\zeta_{p^{n+1}}) \text{ $\mathrm{mod}$ }\mathfrak{a}_{F/K}^{j-s_f}$.

Thus it is enough to check the inequality $j-s_f>er/(p-1)$. Note that $s_f$ is equal to the valuation $v_K(\mathfrak{D}_{K(\zeta_{p^{n+1}})/N})$ of the different of the totally ramified Galois extension $K(\zeta_{p^{n+1}})/N$. To bound this, put $G=\Gal(K(\zeta_{p^{n+1}})/N(\zeta_p))$ and $e'=e(N(\zeta_p)/N)$. We have
\[
v_K(\tau(\varpi)-\varpi))\leq v_K(\tau(\zeta_{p^{n+1}})-\zeta_{p^{n+1}})
\]
for any $\tau\in G$ and thus 
\[
v_K(\mathfrak{D}_{K(\zeta_{p^{n+1}})/N(\zeta_p)})\leq \sum_{\tau\neq 1\in G}v_K(\tau(\zeta_{p^{n+1}})-\zeta_{p^{n+1}})\leq ne.
\]
We also have the equality $v_K(\mathfrak{D}_{N(\zeta_p)/N})=1-1/e'$ and hence we get
\[
s_f=v_K(\mathfrak{D}_{K(\zeta_{p^{n+1}})/N})\leq 1-1/e'+ne.
\]
Since $e'\leq p-1$, the inequality $j-s_f>er/(p-1)$ holds.
\end{proof}

\begin{cor}\label{i}
There exists $g\in G_K$ such that $\eta(\pi_n)\equiv g(\pi_n)\text{ $\mathrm{mod}$ }\beF$ and 
$\eta(\zeta_{p^{n+1}})\equiv g(\zeta_{p^{n+1}})\text{ $\mathrm{mod}$
 }\beF$. 
\end{cor}
\begin{proof}
Let $i\in\bZ$ and $g'\in G_K$ be as in Lemma \ref{prei} and Lemma \ref{z}, respectively. Since $K_n\cap K(\zeta_{p^{n+1}})=K$ (see for example
 \cite[Lemma 5.1.2]{Li}), we can find an element $g\in G_K$ such that
 $g(\pi_n)=\pi_n\zeta_{p^n}^{i}$ and
 $g(\zeta_{p^{n+1}})= g'(\zeta_{p^{n+1}})$.
\end{proof}

\begin{lem}\label{injp}
For $m\in\bZ_{\geq 0}$, set an ideal $\mathfrak{b}_{L_n}^{(m)}$ of $\mathcal{O}_{L_n}$ to be
\[
\mathfrak{b}_{L_n}^{(m)}=\{x\in\mathcal{O}_{L_n}\mid v_K(x)>\frac{er}{p^m(p-1)}\}
\]
and similarly for $F$. Then the $\okey$-algebra homomorphism
 $\eta$ induces an $\okey$-algebra injection
\[
\eta^{(m)}:\mathcal{O}_{L_n}/\mathfrak{b}_{L_n}^{(m)}\to \oef/\mathfrak{b}_{F}^{(m)}
\]
for any $m$.
\end{lem}

\begin{proof}
We may assume that $L_n$ is totally ramified over $K$. We write the Eisenstein polynomial of a uniformizer $\pi_{L_n}$ of
 $L_n$ over $\okey$ as
\[
P(T)=T^{e'}+c_1 T^{e'-1}+\cdots+c_{e'-1} T+c_{e'},
\]
where $e'=e(L_n/K)$. Then $z=\eta(\pi_{L_n})$ satisfies $P(z)=0$ in $\oef/\mathfrak{a}_{F/K}^{j}$. Let $\hat{z}$ be a
 lift of $z$ in $\oef$. Since $j>1$, we have
 $v_K(\hat{z})=1/e'$. The condition $i>e(L_n)r/(p^m(p-1))$ is equivalent
 to the condition
\[
v_K(\hat{z}^i)>\frac{e(L_n)r}{p^m(p-1)}\cdot \frac{1}{e'}=\frac{er}{p^m(p-1)}.
\]
Thus the claim follows.
\end{proof}

Since $L_n$ contains $F_n$, we can consider the ring
\[
\bar{A}_{n,L_n,r+}=W_n(\mathcal{O}_{L_n}/\mathfrak{b}_{L_n})/([\zeta_{p^n}]-1)^rW_n(m_{L_n}/\beLn)
\]
and similarly $\bar{A}_{n,F,r+}$ for $F$. We give these rings structures of $\Sigma$-algebras as follows. The ring $\bar{A}_{n,L_n,r+}$ is
 considered as a $\Sigma$-algebra by using the system
 $\{\pi_n\}_{n\in\mathbb{Z}_{\geq 0}}$ we chose of
 $p$-power roots of $\pi$, as in the previous section. On the other
 hand, using $g\in G_K$ in Corollary \ref{i}, put
 $\tilde{\pi}_n=g(\pi_n)$ and
 $\tilde{\zeta}_{p^{n+1}}=g(\zeta_{p^{n+1}})$. Then we
 consider the ring $\bar{A}_{n,F,r+}$ as a $\Sigma$-algebra by using a
 system of $p$-power roots of $\pi$ containing $\tilde{\pi}_n$. We define $\Fil^r$ and
 $\phi_r$ of these rings in the same way as before.

\begin{lem}\label{Siglinear}
The induced ring homomorphism
\[
\bar{\eta}:\bar{A}_{n,L_n,r+}\to \bar{A}_{n,F,r+}
\] 
is a morphism of the category $'\Mod^{r,\phi}_{/\Sigma}$.
\end{lem}
\begin{proof}
Firstly, we check that $\bar{\eta}$ is $\Sigma$-linear. By definition, this homomorphism commutes with the action of
 the element $u\in\Sigma$. To show the compatibility with the element $Y\in\Sigma$, let us
 consider the commutative diagram
\[
\begin{CD}
W_n(\mathcal{O}_{L_n}/\mathfrak{b}_{L_n}) @>{\eta_n}>>
 W_n(\oef/\mathfrak{b}_F) \\
@VVV @VVV \\
\bar{A}_{n,L_n,r+} @>{\bar{\eta}}>> \bar{A}_{n,F,r+},
\end{CD}
\]
where the horizontal arrows are induced by $\eta$. Note that we have
 $\eta_n([\pi_n])=[\tilde{\pi}_n]$ and $\eta_n([\zeta_{p^{n+1}}])=[\tilde{\zeta}_{p^{n+1}}]$. Let $a\in W(R)^\times$ and
 $v=t/E([\upi])\in W(R)^\times$ be as in the previous section. Let $a_n$ and $v_n$ denote the images of $a$ and
 $v$ in $W_n(\oeln/\beLn)$, respectively. Then the element $v_n$ is a
 solution of the equation
\[
E([\pi_n])v_n=1+[\zeta_{p^{n+1}}]+\cdots+[\zeta_{p^{n+1}}]^{p-1}.
\]
Similarly, we define elements $\tilde{a}_n$ and $\tilde{v}_n$ of $W_n(\oef/\beF)$ using
 $\tilde{\pi}_n$ and $\tilde{\zeta}_{p^{n+1}}$. By definition, the element
 $\tilde{v}_n$ is a solution of the equation
\[
E([\tilde{\pi}_n])\tilde{v}_n=1+[\tilde{\zeta}_{p^{n+1}}]+\cdots+[\tilde{\zeta}_{p^{n+1}}]^{p-1}.
\]
Now what we have to show is the equality
\[
\bar{\eta}(a_n v_{n}^{-1}E([\pi_n])^{p-1})=\tilde{a}_n\tilde{v}_{n}^{-1}E([\tilde{\pi}_n])^{p-1}
\]
in the ring $\bar{A}_{n,F,r+}$.
Since the element $a_n$ of $W_n(\oeln/\beLn)$ is a linear combination
 of the elements $1,[\zeta_{p^{n+1}}],\ldots,[\zeta_{p^{n+1}}]^{p-1}$
 over $\mathbb{Z}$, we have $\bar{\eta}(a_n)=\tilde{a}_n$ in
 $\bar{A}_{n,F,r+}$. The elements $\tilde{v}_n$ and
 $\bar{\eta}(v_n)$ satisfy the same equation in
 $\bar{A}_{n,F,r+}$. Since these two elements are invertible, we get
 $\bar{\eta}(v_n)^{-1}E([\tilde{\pi}_n])=\tilde{v}_{n}^{-1}E([\tilde{\pi}_n])$ and the equality holds. Since the diagram above is compatible with the Frobenius endomorphisms, we see from the definition that $\bar{\eta}$ also preserves $\Fil^r$
 and commutes with $\phi_r$ of both sides.
\end{proof}

Thus we get a homomorphism of abelian groups
\[
T^{*}_{\mathrm{crys},L_n,\pi_n}(M)\to T^{*}_{\mathrm{crys},F,\tilde{\pi}_n}(M).
\]
Then the following lemma, whose proof is omitted in \cite[Subsection
 3.13]{Ab}, implies that this homomorphism is an injection. We insert
 here a proof of this lemma for the convenience of the reader.
\begin{lem}
The ring homomorphism $\bar{\eta}:\bar{A}_{n,L_n,r+}\to\bar{A}_{n,F,r+}$ is
 an injection.
\end{lem}
\begin{proof}
Let $x=(x_0,\ldots,x_{n-1})$ be an element of $W_n(\oeln/\beLn)$ such that 
\[
(\eta^{(0)}(x_0),\ldots,\eta^{(0)}(x_{n-1}))\in([\zeta_{p^n}]-1)^rW_n(m_F/\beF),
\]
where $\eta^{(0)}$ is as in Lemma \ref{injp}. Suppose that $x_0=\cdots=x_{m-1}=0$ for some $0\leq m \leq n-1$. Let $\hat{z}_i\in\oef$ be a lift of $\eta^{(0)}(x_i)$. By Lemma \ref{Wittlem}, we have
\[
(0,\ldots,0,\hat{z}_m,\ldots,\hat{z}_{n-1})=([\zeta_{p^n}]-1)^r(\hat{y}_0,\ldots,\hat{y}_{n-1})
\]
for some $\hat{y}_0,\ldots,\hat{y}_{n-1}\in m_F$. Thus we get $\hat{y}_0=\cdots=\hat{y}_{m-1}=0$ and $v_K(\hat{z}_m)>er/(p^{n-1-m}(p-1))$. Then Lemma \ref{injp} implies that $x_m$ is contained in the ideal $\mathfrak{b}_{L_n}^{(n-1-m)}/\beLn$ and 
\[
x=([\zeta_{p^n}]-1)^r(0,\ldots,0,y,0,\ldots,0)+(0,\ldots,0,x'_{m+1},\ldots,x'_{n-1})
\]
for some $y\in m_{L_n}/\beLn$ and $x'_{m+1},\ldots,x'_{n-1}\in \oeln/\beLn$. Repeating this, we see that $x$ is zero in $\bar{A}_{n,L_n,r+}$ and the lemma follows.
\end{proof}

Now Corollary \ref{Fon} shows that the abelian group
 $T^{*}_{\mathrm{crys},L_n,\pi_n}(M)$ has the same cardinality as $T^{*}_{\mathrm{crys},\kbar,\pi_n}(M)$. This implies that the abelian group
 $T^{*}_{\mathrm{crys},F,\tilde{\pi}_n}(M)$ has cardinality no less than $\# T^{*}_{\mathrm{crys},\kbar,\pi_n}(M)$. Let $g\in G_K$ be as in Corollary \ref{i}. Then we have the following lemma.

\begin{lem}
The $G_{F_n}$-module $T^{*}_{\mathrm{crys},\kbar,\tilde{\pi}_n}(M)$ is isomorphic to the conjugate of the $G_{F_n}$-module $T^{*}_{\mathrm{crys},\kbar,\pi_n}(M)$ by the element
 $g$. 
\end{lem}

\begin{proof}
Let us consider the composite
\[
\Sigma\to\bar{A}_{n,r+}\overset{g}{\to} \bar{A}_{n,r+}
\]
of the ring homomorphism defined by $u\mapsto [\pi_n]$ and the map induced by $g$. We
 can check that this is the natural ring homomorphism defined by
 $u\mapsto [\tilde{\pi}_n]$ as in the proof of Lemma \ref{Siglinear}. Thus we have an
 isomorphism of abelian groups
\begin{align*}
\Hom_{\Sigma}(M,\bar{A}_{n,r+})&\to
 \Hom_{\Sigma}(M,\bar{A}_{n,r+}) \\
f &\mapsto g\circ f,
\end{align*}
where we consider on the ring $\bar{A}_{n,r+}$ on the right-hand side the filtered $\phi_r$-module structure over $\Sigma$ defined by
 $\tilde{\pi}_n$. 
As in the proof of Lemma \ref{Siglinear}, we can check that this
 isomorphism induces an injection
\[
\Hom_{\Sigma,\Fil^r,\phi_r}(M,\bar{A}_{n,r+})\to
 \Hom_{\Sigma,\Fil^r,\phi_r}(M,\bar{A}_{n,r+}).
\]
This is also an isomorphism, for the map $f\mapsto g^{-1}\circ f$ defines its inverse. 
\end{proof}

Thus we have $\# T^{*}_{\mathrm{crys},\kbar,\tilde{\pi}_n}(M)=\# T^{*}_{\mathrm{crys},\kbar,\pi_n}(M)$.
Since $L_n$ is Galois over $K$, this lemma also shows that 
the finite
Galois extension of $F_n$ cut out by the action on
 $T^{*}_{\mathrm{crys},\kbar,\tilde{\pi}_n}(M)$ is $L_n$. Hence we
 see from Corollary \ref{Fon} that $F$ contains $L_n$ and
 Proposition \ref{Pm'} follows. This concludes the proof of Theorem \ref{main_br}.
\end{proof}

\begin{proof}[Proof of Corollary \ref{coret}]

The second assertion follows immediately from Theorem \ref{main_br} and \cite[Th\'{e}or\`{e}me 1.1]{Ca2}. As for the first assertion, note that if $r=0$ then $V$ is unramified and the assertion is trivial. Thus we may assume $p\geq 3$. Since we have the natural surjection $\mathcal{L}/p^n\mathcal{L}\to \mathcal{L}/\mathcal{L}'$, we may also assume $\mathcal{L}'=p^n\mathcal{L}$. For $\hat{\cM}\in\Mod^{r,\phi,N}_{/S}$, let us consider the $G_K$-module
\[
T_{\mathrm{st},\upi}^{*}(\hat{\cM})=\Hom_{S,\Fil^r,\phi_r,N}(\hat{\cM},\Ast).
\]
By \cite[Theorem 2.3.5]{Li}, there exists $\hat{\cM}\in\Mod^{r,\phi,N}_{/S}$ such that the $G_K$-module $\mathcal{L}$ is isomorphic to $T_{\mathrm{st},\upi}^{*}(\hat{\cM})$. Then we see that the $G_K$-module $\mathcal{L}/p^n\mathcal{L}$ is isomorphic to $\Tsta(\hat{\cM}/p^n\hat{\cM})$ and the assertion follows from Theorem \ref{main_br}.
\end{proof}

\begin{rmk}\label{sharp}
The ramification bound in Theorem \ref{main_br} is sharp for
 $r\leq 1$. Indeed, the greatest upper ramification break
 $1+e(n+1/(p-1))$ for $r=1$
 is obtained by the $p^n$-torsion of the Tate curve
 $\kbar^\times/\pi^\mathbb{Z}$ (see Remark \ref{Tate}). The author does not know whether these
 bounds are sharp also for $r\geq 2$.
\end{rmk}

\end{document}